\documentclass{article}
\usepackage[a4paper, total={6in, 8in}]{geometry}
\usepackage{anyfontsize}
\usepackage[T1]{fontenc}
\usepackage[utf8]{inputenc}
\usepackage{amsthm}
\usepackage{amsmath}
\usepackage{amssymb}
\usepackage{mathtools}
\newtheorem{theorem}{Theorem}[section]
\newtheorem{hypothesis}[theorem]{Hypothesis}
\newtheorem{definition}[theorem]{Definition}
\newtheorem{proposition}[theorem]{Proposition}
\newtheorem{lemma}[theorem]{Lemma}
\newtheorem{remark}[theorem]{Remark}
\let\originalleft\left
\let\originalright\right
\renewcommand{\left}{\mathopen{}\mathclose\bgroup\originalleft}
\renewcommand{\right}{\aftergroup\egroup\originalright}
\newcommand{\footremember}[2]{%
   \footnote{#2}
    \newcounter{#1}
    \setcounter{#1}{\value{footnote}}%
}
\newcommand{\footrecall}[1]{%
    \footnotemark[\value{#1}]%
} 

\begin{document}
\title{The Stochastic TR-BDF2 Scheme of Order 2}
\author{Tom{\'a}s Caraballo\footremember{1}{Departamento de Ecuaciones Diferenciales y An{\'a}lisis Num{\'e}rico, Universidad de Sevilla; caraball@us.es \& macarena@us.es \& iroldan@us.es} \and Macarena G{\'o}mez-M{\'a}rmol\footrecall{1} \and Ignacio Rold{\'a}n\footrecall{1}}
\date{\today}
\maketitle
\begin{abstract}
Our main objective in this paper is to develop a second-order stochastic numerical method which generalizes the well-known deterministic TR-BDF2 scheme. Since most stochastic techniques used for approximating the solution of a stochastic differential equation may have lower order compared to the deterministic case, we have elaborated a scheme which not only preserves the second-order accuracy of the original scheme in the stochastic framework, but also its $A$-stability. Once we obtain the scheme and prove its second-order accuracy and $A$-stability, which is not a trivial task, we also state a result concerning its $MS$-stability. This concept is also analyzed for different parameter ranges in our scheme and the It{\^o}--Taylor approximation of order 2, revealing scenarios where, for certain time step sizes, the developed method is $MS$-stable while the It{\^o}--Taylor one is not. This concept is really useful to tackle slow-fast problems such as stiff ones, which we aim to explore further in future work. Finally, we validate the theoretical results with some academic test cases.
\end{abstract}
\paragraph{Subjclass:} 60H10, 65C20, 65P99.

\section*{Introduction}
Nowadays, the concern of researchers to simulate more and more realistic problems has led the scientific community to become interested in problems with a significant randomness component. In this line, one of the ways to take this randomness into account is to consider stochastic problems. In our case, we are interested in ordinary differential systems where white noise has been added, since it represents many real situations such as population dynamics, neuroscience behavior, chemical reactions, etc.

These problems have started to be studied from the point of view of Mathematical Analysis some years ago and we can find a wide bibliography related to these studies \cite{Buckwar2021, Hairer2010, Lord2014}. Since most of these models have no known explicit solutions, numerical analysis of SDEs is essential for attempting to study the behavior of those solutions \cite{Ruemelin1982}. Some recent important works on stochastic systems include \cite{Yuan2024, Yuan2025, Zhu2025}.

The numerical methods developed are mainly based on reproducing in some way the well-known ones for the deterministic case. Thus, for example, based on the It{\^o} analysis and the It{\^o}--Taylor expansion, deterministic methods based on Taylor's development are generalized \cite{Ito1951_1, Ito1951_2}. The best known amongst these methods are the Euler-Maruyama \cite{Cambanis1996} and the Milstein ones \cite{Maruyama1955}. Related developments in the numerical treatment of stochastic systems can be found in \cite{Kloeden2025, Li2025, Metzger2025, Printems2001}.

Other generalizations have been developed, such as the corresponding $\theta$-scheme, which is the stochastic $\theta$-Maruyama \cite{Berkolaiko2012} or a generalization of the implicit BDF2 scheme \cite{Andersson2017} and even a stochastic partial differential equation version \cite{Kruse2023}. However, these schemes are of order 0.5 in time, which might not be sufficient to obtain good precision in certain models. On the other hand, schemes with higher order have been elaborated, like Adams-Bashforth \cite{Ewald2005} with order 2, Runge-Kutta with orders up to 1.5 \cite{Burrage1996, R2010}, and approximations of weak order 2 in \cite{Ninomiya2008, R2009}.

Within the field of differential systems we are interested in studying slow-fast problems, meaning that the speeds at which the different variables move are not comparable in magnitude. This type of systems is included in what are known as stiff problems. Over the recent years, more elaborate schemes have been developed in an efficient way to tackle this kind of problems and address the challenges. An example of that is the so-called splitting techniques \cite{Berkaoui2008, Buckwar2022, Wang2010}, which have order 1 at most. However, constructing second order stochastic methods is challenging because the order may be reduced compared to the deterministic case \cite{Kloeden1992, Mao1997}.

To solve these problems properly, it is necessary to choose the time step for reasons of stability and not for reasons of accuracy. Therefore, it is important to have methods that have good numerical properties. It may happen that these properties can be achieved but with an extremely small time step which makes the method unfeasible. In this paper, we focus on the development of a scheme and the study of its stability properties in the scalar case, which provides a natural framework in which key ideas can be explored and validated. In \cite{Buckwar2010}, a stability study is conducted in the context of a system; this setting is likewise of significant interest and will be investigated in subsequent studies.

Our aim in this paper is to generalize the deterministic TR-BDF2 method to the stochastic framework, so that it maintains all the good properties and in particular the order 2 of convergence.

The TR-BDF2 in deterministic version has been studied for example in \cite{Bonaventura2017, Hosea1996, Bank1985}. This method has been used in many fields of science with good results.

The remainder of the paper is organized as follows. Section~\ref{sec1} contains the problem we want to approximate and the concept of order of convergence that we are going to use when approximating the problem to be solved. Section~\ref{s3} is devoted to a short summary of the It{\^o} analysis, the It{\^o}--Taylor expansion and the usual numerical methods that are obtained from the expansion. Sections~\ref{s4} and \ref{s5} actually contain the bulk of the work. In the first one, the method is constructed and the convergence order is studied from a theoretical point of view, following an argument similar to that in \cite{Ewald2005}, but adapted to our setting, which involves a non-uniform partition and leverages the implicitness of the scheme, among other things. Section~\ref{s5} is devoted to the analysis of the numerical stability of the constructed method, a very important property to attack stiff problems. Finally, in Section~\ref{numerical}, some validation tests of the theoretical results presented in the previous sections are shown.

\section{Problem Formulation}\label{sec1}
In this first section, we present the continuous problem for which we are interested in developing numerical methods that allow us to approximate its solution. We will also define the concept of numerical method of convergence order $p$ for the stochastic case.

Let $\left(\Omega, \mathcal{F}, \left\{\mathcal{F}_{t}\right\}_{t \geq 0}, P\right)$ be a complete probability space with the natural filtration $\left\{\mathcal{F}_{t}\right\}_{t \geq 0}$. Consider a $d$-dimensional stochastic differential equation (SDE)
\begin{equation}
\left\{\begin{array}{ll}
\displaystyle {\rm d}X_{t} = a\left(t, X_{t}\right){\rm d}t + \sum_{j = 1}^{m}b^{j}\left(t, X_{t}\right){\rm d}W_{t}^{j}, & \displaystyle t \in \left[0, T\right],\\
\displaystyle X_{0} \text{ is given}, & \displaystyle
\end{array}\right.
\label{est}
\end{equation}
\noindent where $T$ is a positive constant, $a : \left[0, T\right] \times \mathbb{R}^{d} \to \mathbb{R}^{d}$, $b^{j} : \left[0, T\right] \times \mathbb{R}^{d} \to \mathbb{R}^{d}$, with $j = 1, \ldots, m$ and $X_{0}$ is the initial data. The processes $W_{t}^{j}$ are independent scalar Wiener processes adapted to $\left\{\mathcal{F}_{t}\right\}_{t \geq 0}$, and the initial data $X_{0}$ is independent of them, satisfying $E\left[\left\|X_{0}\right\|^{2}\right] < \infty$, where $\left\|\cdot\right\|$ denotes the Euclidean norm.

We now present a result establishing the existence and uniqueness of the solution to the above problem \cite{Kloeden1992}:

\begin{theorem}\label{eu}
Suppose that
\begin{itemize}
\item[(A.1)] $a$ and $b^{j}$, for all $j = 1, \ldots, m$, are jointly measurable in $\left(t, x\right) \in \left[0, T\right] \times \mathbb{R}^{d}$ (measurability);

\item[(A.2)] There exists a constant $K > 0$ such that
\begin{equation*}
\left\|a\left(t, x\right) - a\left(t, y\right)\right\|^{2} \vee \left\|b^{j}\left(t, x\right) - b^{j}\left(t, y\right)\right\|^{2} \leq K\left\|x - y\right\|^{2},
\end{equation*}
\noindent for all $j = 1, \ldots, m$, for all $t \in \left[0, T\right]$ and $x, y \in \mathbb{R}^{d}$ (Lipschitz condition);

\item[(A.3)] There exists a constant $K > 0$ such that
\begin{equation*}
\left\|a\left(t, x\right)\right\|^{2} \vee \left\|b^{j}\left(t, x\right)\right\|^{2} \leq K\left(1 + \left\|x\right\|^{2}\right),
\end{equation*}
\noindent for all $j = 1, \ldots, m$, for all $t \in \left[0, T\right]$ and $x \in \mathbb{R}^{d}$ (linear growth bound).
\end{itemize}

Then, the stochastic differential equation \eqref{est} has a path-wise unique strong solution $X_{t}$ on $\left[0, T\right]$ with $E\left[\sup_{0 \leq t \leq T}\left\|X_{t}\right\|^{2}\right] < \infty$.
\end{theorem}

In the above theorem, the symbol $\vee$ represents the maximum between two quantities, which implies that the inequality is satisfied by both the drift and the diffusion terms. In the following, ${\rm C}^{p, q}$ denotes the class of functions that are $p$ times continuously differentiable with respect to the first variable and $q$ times continuously differentiable with respect to the second variable.

\begin{remark}\label{rm11}
The assumption of global Lipschitz continuity can be weakened to local Lipschitz continuity, and the linear growth condition can be replaced by the inequality
\begin{equation*}
\left\langle a\left(t, x\right), x\right\rangle \vee \left\|b^{j}\left(t, x\right)\right\|^{2} \leq K\left(1 + \left\|x\right\|^{2}\right),
\end{equation*}
\noindent for some constant $K > 0$, for all $j = 1, \ldots, m$, for all $x \in \mathbb{R}^{d}$ and $t \in \left[0, T\right]$ (see \cite{Mao1997}).

In particular, if the drift and each diffusion coefficient are continuously differentiable with bounded derivatives, i.e., they belong to the set that we denote by ${\rm C}^{1, 1}$, then the measurability condition and local Lipschitz continuity are satisfied, which guarantees existence and uniqueness of the solution. We shall work under this latter framework, as will be shown later.
\end{remark}

Let us now look at the concept of order of convergence in the stochastic case. We shall denote by $L^{2}\left(\Omega, \mathbb{R}^{d}\right)$ to the space whose processes satisfy that $E\left[\left\|\cdot\right\|^{2}\right] < \infty$.

\begin{definition}\label{def}
We consider a partition $\Delta_{N} = \left\{t_{0} = 0 < t_{1} < \ldots < t_{N} = T\right\}$ of $\left[0, T\right]$ and let $Y_{n}$ an approximation of the solution of problem \eqref{est} on the point $t_{n}$, for all $n = 0, \ldots, N$. Then, we say that $\left\{Y_{n}\right\}_{0 \leq n \leq N}$ converges (strongly) to $X_{t}$, the solution of that SDE, with order $p > 0$ if:
\begin{equation*}
\sup_{0 \leq n \leq N}\left\|X_{t_{n}} - Y_{n}\right\|_{L^{2}\left(\Omega, \mathbb{R}^{d}\right)} \coloneqq \sup_{0 \leq n \leq N}\left(E\left[\left\|X_{t_{n}} - Y_{n}\right\|^{2}\right]\right)^{\frac{1}{2}} \leq \mathcal{O}\left(h^{p}\right),
\end{equation*}
\noindent where $h$ is the diameter of the partition.
\end{definition}

\begin{remark}
The concept of strong convergence in this framework is used whenever we are trying to approximate the solution of the differential system. If we are interested in calculating the approximation of moments, then we would have to use the concept of weak convergence. This convergence is studied in depth in references \cite{Ninomiya2008, R2009}.
\end{remark}

\section{The It{\^o}--Taylor analysis}\label{s3}
In this section we introduce a generalization of the Taylor expansion in the stochastic context and some numerical schemes obtained from such expansion.

\subsection{The It{\^o}--Taylor expansion}
First, we shall consider an It{\^o} process $X_{t}$ that satisfies the $d$-dimensional SDE defined in \eqref{est}. Henceforth, we take the same notation that is used in \cite{Kloeden1992}: Let $\alpha = \left(j_{1}, \ldots, j_{l}\right)$ be a multi-index, we define $\mathcal{M}$ as the set of all multi-indices:
\begin{equation*}
\mathcal{M} = \left\{\alpha = \left(j_{1}, \ldots, j_{l}\right) : j_{i} \in \left\{0, 1, \ldots, m\right\}, i \in \left\{1, \ldots, l\right\}, \text{ for } l = 1, 2, \ldots\right\} \cup \left\{v\right\},
\end{equation*}
\noindent where $l \coloneqq l\left(\alpha\right)$ represent the length of a multi-index $\alpha$ and $v$ is the multi-index of length zero. Also, let $n\left(\alpha\right)$ be the number of components of a multi-index $\alpha$ which value is equal to 0, ${}^{-}\alpha$ the multi-index outcome of removing the initial element of $\alpha$ and $\alpha^{-}$ the multi-index outcome of removing the last element of $\alpha$.

Taking $a^{k}$ and $b^{k, j}$ as the component $k$ of $a$ and $b^{j}$ respectively with $k = 1, \ldots, d$, we define the following operators:
\begin{equation*}
L^{0} = \frac{\partial}{\partial t} + \sum_{k = 1}^{d}a^{k}\frac{\partial}{\partial x_{k}} + \frac{1}{2}\sum_{k, l = 1}^{d}\sum_{j = 1}^{m}b^{k, j}b^{l, j}\frac{\partial^{2}}{\partial x_{k}\partial x_{l}},
\end{equation*}
\begin{equation*}
L^{j} = \sum_{k = 1}^{d}b^{k, j}\frac{\partial}{\partial x_{k}}, j = 1, \ldots, m.
\end{equation*}

Then, we denote ${}_{\alpha}f$ as the coefficient functions of $f : \left[0, T\right] \times \mathbb{R}^{d} \to \mathbb{R}^{d}$, where ${}_{v}f \equiv f$ and ${}_{\alpha}f$ is the result of applying recursively the operators to $f$ as:
\begin{equation*}
{}_{\alpha}f = L^{j_{1}}\left({}_{{}^{-}\alpha}f\right).
\end{equation*}

On the other hand, we define the multiple It{\^o} integral recursively as
\begin{equation*}
\left\{\begin{array}{l}
\displaystyle I_{v}\left[f\left(\cdot, X_{\cdot}\right)\right]_{s, t} \equiv f\left(t, X_{t}\right), {\rm d}W_{t}^{0} = {\rm d}t,\\
\displaystyle I_{\alpha}\left[f\left(\cdot, X_{\cdot}\right)\right]_{s, t} = \int_{s}^{t}I_{\alpha^{-}}\left[f\left(\cdot, X_{\cdot}\right)\right]_{s, r}{\rm d}W_{r}^{j_{l}},
\end{array}\right.
\end{equation*}
\noindent and we denote by $I_{\alpha}\left[f\left(\cdot, X_{\cdot}\right)\right]_{s, t} = I_{\alpha, s, t}$ if $f \equiv 1$.

We say that the above integral exists if the adapted process $f = \left\{f\left(t\right), t \geq 0\right\}$ is right-continuous, possesses left hand limits (c{\`a}dl{\`a}g) and belongs to $\mathcal{H}_{\alpha}$, where these sets are defined below:

\begin{itemize}
\item $f \in \mathcal{H}_{v}$ if for each $t \geq 0$, $\left\|f\left(t\right)\right\| < \infty$ with probability 1 (w.p.1).

\item $f \in \mathcal{H}_{\left(0\right)}$ if for each $t \geq 0$, $\int_{0}^{t}\left\|f\left(s\right)\right\|{\rm d}s < \infty$ (w.p.1).

\item $f \in \mathcal{H}_{\left(1\right)}$ if for each $t \geq 0$, $\int_{0}^{t}\left\|f\left(s\right)\right\|^{2}{\rm d}s < \infty$ (w.p.1).

In addition, we write $\mathcal{H}_{\left(j\right)} = \mathcal{H}_{\left(1\right)}$ for each $j \in \left\{2, \ldots, m\right\}$ if $m \geq 2$.

\item $f \in \mathcal{H}_{\alpha}$, with $l\left(\alpha\right) \geq 2$, if for each $t \geq 0$, $I_{\alpha^{-}}\left[f\left(\cdot\right)\right]_{0, t} \in \mathcal{H}_{\left(j_{l}\right)}$ (w.p.1).
\end{itemize}

Now, we introduce the concept of hierarchical set and remainder set, which are the sets where the It{\^o}--Taylor expansion is well-defined.

\begin{definition}
$\mathcal{A} \subset \mathcal{M}$ is said to be a hierarchical set if:

\begin{itemize}
\item $\mathcal{A} \neq \emptyset$.

\item $\sup_{\alpha \in \mathcal{A}}l\left(\alpha\right) < \infty$.

\item ${}^{-}\alpha \in \mathcal{A}$ for each $\alpha \in \mathcal{A} \setminus \left\{v\right\}$.
\end{itemize}
\end{definition}

\begin{definition}
We define $\mathcal{B}\left(\mathcal{A}\right)$, the remainder set of $\mathcal{A}$, as the following set:
\begin{equation*}
\mathcal{B}\left(\mathcal{A}\right) = \left\{\alpha \in \mathcal{M} \setminus \mathcal{A} : {}^{-}\alpha \in \mathcal{A}\right\}.
\end{equation*}
\end{definition}

Once these definitions have been introduced, we can define the It{\^o}--Taylor expansion.

\begin{definition}
The It{\^o}--Taylor expansion of a function $f : \left[0, T\right] \times \mathbb{R}^{d} \to \mathbb{R}^{d}$ from $s$ to $t$, $0 \leq s \leq t \leq T$, in the hierarchical set $\mathcal{A}$, provided all derivatives of $f$, $a$ and $b^{j}$ and all of the multiple It{\^o} integrals exist, is defined as
\begin{equation*}
f\left(t, X_{t}\right) = \sum_{\alpha \in \mathcal{A}}{}_{\alpha}f\left(s, X_{s}\right)I_{\alpha, s, t} + \sum_{\alpha \in \mathcal{B}\left(\mathcal{A}\right)}I_{\alpha}\left[{}_{\alpha}f\left(\cdot, X_{\cdot}\right)\right]_{s, t}.
\end{equation*}
\end{definition}

In particular, if we take $p$ a number which can be written like $p = 0.5k$, with $k \in \mathbb{N} \cup \left\{0\right\}$, we define the set
\begin{equation*}
\mathcal{A}_{p} = \left\{\alpha \in \mathcal{M} : l\left(\alpha\right) + n\left(\alpha\right) \leq 2p \text{ or } l\left(\alpha\right) = n\left(\alpha\right) = p + \frac{1}{2}\right\}.
\end{equation*}

Then $\mathcal{A}_{p}$ is a hierarchical set, and the It{\^o}--Taylor expansion in this case is known as It{\^o}--Taylor expansion of order $p$ and its expression is
\begin{equation}
f\left(t, X_{t}\right) = \sum_{\alpha \in \mathcal{A}_{p}}{}_{\alpha}f\left(s, X_{s}\right)I_{\alpha, s, t} + \sum_{\alpha \in \mathcal{B}\left(\mathcal{A}_{p}\right)}I_{\alpha}\left[{}_{\alpha}f\left(\cdot, X_{\cdot}\right)\right]_{s, t}.
\label{Ito-Taylor}
\end{equation}

Consequently, when $s = t_{n}$ and $t = t_{n + 1}$, we have:
\begin{equation*}
X_{t_{n + 1}} = \sum_{\alpha \in \mathcal{A}_{p}}{}_{\alpha}F\left(t_{n}, X_{t_{n}}\right)I_{\alpha, n} + \sum_{\alpha \in \mathcal{B}\left(\mathcal{A}_{p}\right)}I_{\alpha}\left[{}_{\alpha}F\left(\cdot, X_{\cdot}\right)\right]_{t_{n}, t_{n + 1}},
\end{equation*}
\noindent where ${}_{\alpha}F\left(t, x\right)$ are the coefficient functions of $F\left(t, x\right) \equiv x$ and we have denoted $I_{\alpha, n} = I_{\alpha, t_{n}, t_{n + 1}}$. From now on, we denote the remainder of the It{\^o}--Taylor expansion $\sum_{\alpha \in \mathcal{B}\left(\mathcal{A}_{p}\right)}I_{\alpha}\left[{}_{\alpha}f\left(\cdot, X_{\cdot}\right)\right]_{s, t}$ by $\tilde{R}_{p}\left[f\right]_{s, t}$, and the stochastic terms of the It{\^o}--Taylor expansion of solution $X_{t}$ of order $p$ from $s$ to $t$, i.e., $\sum_{\alpha \in \mathcal{A}_{p} : l\left(\alpha\right) \neq n\left(\alpha\right)}{}_{\alpha}F\left(s, X_{s}\right)I_{\alpha, s, t}$, by $S_{p}\left[X_{s}\right]_{s, t}$.

Finally, in order to ensure the approximation to be of order $p$, we state the standard hypotheses required for the coefficient functions ${}_{\alpha}F$, which can be related with those of Theorem~\ref{eu} of existence and uniqueness.

\begin{hypothesis}
Assume that ${}_{\alpha}F$ satisfies that:

\begin{itemize}
\item[(H.1)] ${}_{{}^{-}\alpha}F \in {\rm C}^{1, 2}$ and ${}_{\alpha}F \in \mathcal{H}_{\alpha}$, for all $\alpha \in \mathcal{A}_{p} \cup \mathcal{B}\left(\mathcal{A}_{p}\right)$ (measurability condition).

\item[(H.2)] There exists a constant $C_{1} > 0$ such that $\left\|{}_{\alpha}F\left(t, x\right) - {}_{\alpha}F\left(t, y\right)\right\|^{2} \leq C_{1}\left\|x - y\right\|^{2}$, for all $\alpha \in \mathcal{A}_{p}$, $t \in \left[0, T\right]$ and $x, y \in \mathbb{R}^{d}$ (Lipschitz condition with respect to the second variable).

\item[(H.3)] There exists a constant $C_{2} > 0$ such that $\left\|{}_{\alpha}F\left(t, x\right)\right\|^{2} \leq C_{2}\left(1 + \left\|x\right\|^{2}\right)$, for all $\alpha \in \mathcal{A}_{p} \cup \mathcal{B}\left(\mathcal{A}_{p}\right)$, $t \in \left[0, T\right]$ and $x \in \mathbb{R}^{d}$ (lineal growth bound condition).
\end{itemize}
\end{hypothesis}

As discussed in Remark~\ref{rm11}, where ${\rm C}^{1, 1}$ regularity was shown to imply local Lipschitz continuity in the relaxation of the existence and uniqueness result of Theorem~\ref{eu}, Hypothesis~(H.2) can likewise be relaxed for the drift when the structure of the scheme allows it. The following condition will be used later to establish convergence results under weaker assumptions.

\begin{hypothesis}\label{H}
Suppose that

\item[(H.1$^{\ast}$) Linear growth and regularity.] The functions ${}_{{}^{-}\alpha}F \in {\rm C}^{1, 2}$, ${}_{\alpha}F \in \mathcal{H}_{\alpha}$ and they satisfy global linear growth bounds, for $\alpha \in \mathcal{A}_{p} \cup \mathcal{B}\left(\mathcal{A}_{p}\right)$. Also, the functions ${}_{\alpha}F$, for $\alpha \in \mathcal{A}_{p}$, with $\alpha \neq \left(0\right)$, are globally Lipschitz.

\item[(H.2$^{\ast}$) One-sided Lipschitz condition and one-sided linear growth.] There exists $C \geq 0$ such that for all $x, y \in \mathbb{R}^{d}$ and $t \in \left[0, T\right]$,
\begin{equation}\label{Lips}
\left\langle a\left(t, x\right) - a\left(t, y\right), x - y\right\rangle \leq C\left\|x - y\right\|^{2},
\end{equation}
\noindent and there exists $K > 0$ such that for all $x \in \mathbb{R}^{d}$ and $t \in \left[0, T\right]$,
\begin{equation}\label{glboun}
\left\langle a\left(t, x\right), x\right\rangle \leq K\left(1 + \left\|x\right\|^{2}\right),
\end{equation}
\end{hypothesis}

Similar to the construction of numerical methods for deterministic differential equations based on the Taylor development, we can construct methods in the stochastic case by taking the It{\^o}--Taylor expansion as we will see in the next subsection.

\subsection{Numerical method based on It{\^o}--Taylor expansion}
Taking the It{\^o}--Taylor expansion of order $p$ and disregarding the summation corresponding to the remainder set, we obtain a whole family of numerical methods that can be written as
\begin{equation*}
\left\{\begin{array}{l}
\displaystyle \text{Given } Y_{0}, \text{ compute } \forall n = 0, \ldots, N - 1,\\
\displaystyle Y_{n + 1} = \sum_{\alpha \in \mathcal{A}_{p}}{}_{\alpha}F\left(t_{n}, Y_{n}\right)I_{\alpha, n}.
\end{array}\right.
\end{equation*}

In the particular case $p = 0.5$, we have $\mathcal{A}_{0.5} = \left\{v, \left(0\right), \left(1\right), \ldots, \left(m\right)\right\}$ and the scheme is written as:
\begin{equation*}
\left\{\begin{array}{l}
\displaystyle \text{Given } Y_{0}, \text{ compute } \forall n = 0, \ldots, N - 1,\\
\displaystyle Y_{n + 1} = Y_{n} + a\left(t_{n}, Y_{n}\right)h_{n} + \sum_{j = 1}^{m}b^{j}\left(t_{n}, Y_{n}\right)I_{\left(j\right), n},
\end{array}\right.
\end{equation*}
\noindent where $h_{n} = t_{n + 1} - t_{n}$. This method is known as Euler-Maruyama scheme, which has been studied by different authors. In \cite{Cambanis1996}, we can also find the proof that this is a method of order $p = 0.5$.

For the case $p = 1$, the hierarchical set is $\mathcal{A}_{1} = \mathcal{A}_{0.5} \cup \left\{\left(1, 1\right), \left(1, 2\right), \ldots, \left(m, m\right)\right\}$ and the scheme is:
\begin{equation*}
\left\{\begin{array}{l}
\displaystyle \text{Given } Y_{0}, \text{ compute } \forall n = 0, \ldots, N - 1,\\
\displaystyle Y_{n + 1} = Y_{n} + a\left(t_{n}, Y_{n}\right)h_{n} + \sum_{j = 1}^{m}b^{j}\left(t_{n}, Y_{n}\right)I_{\left(j\right), n} + \sum_{j_{1}, j_{2} = 1}^{m}L^{j_{1}}b^{j_{2}}\left(t_{n}, Y_{n}\right)I_{\left(j_{1}, j_{2}\right), n}.
\end{array}\right.
\end{equation*}

This is the Milstein scheme and a proof of order of convergence $p = 1$ can be found in \cite{Kloeden1992}. We have presented here the two best known and most common methods in the numerical solution of ordinary differential equations. In the same way that we have deduced these methods, other methods can be obtained by considering other values of $p$.

\section{The stochastic TR-BDF2 method}\label{s4}
Before developing the stochastic TR-BDF2 method, let us recall the deterministic version of the method as well as its main numerical properties.

Let us consider a $d$-dimensional differential equation
\begin{equation*}
\left\{\begin{array}{ll}
\displaystyle x'\left(t\right) = a\left(t, x\left(t\right)\right), & \displaystyle t \in \left[0, T\right],\\
\displaystyle x\left(0\right) = x_{0}, & \displaystyle
\end{array}\right.
\end{equation*}
\noindent where $T$ is a positive constant, $a : \left[0, T\right] \times \mathbb{R}^{d} \to \mathbb{R}^{d}$ and $x_{0}$ is the initial data. Our objective is to solve this problem from a numerical point of view, i.e., we want to construct one sequence $\left\{y_{n}\right\} \simeq \left\{x\left(t_{n}\right)\right\}$, $n = 0, \ldots, N$, for some $N \in \mathbb{N}$.

For simplicity, we consider a uniform partition of the interval $\left[0, T\right]$ with step $h = \frac{T}{N}$:
\begin{equation}
\Delta_{N} = \left\{t_{0} = 0 < t_{1} < \cdots < t_{N} = T\right\}.
\label{Dn}
\end{equation}

The method TR-BDF2 for this problem is written as:
\begin{equation}
\left\{\begin{array}{l}
\displaystyle \text{Given } y_{0}, \text{ compute } \forall n = 0, \ldots, N - 1,\\
\displaystyle y_{n + \gamma} = y_{n} + \frac{1}{2}\left(a\left(t_{n}, y_{n}\right) + a\left(t_{n + \gamma}, y_{n + \gamma}\right)\right)h\gamma,\\
\displaystyle y_{n + 1} = \gamma_{3}y_{n + \gamma} + \left(1 - \gamma_{3}\right)y_{n} + a\left(t_{n + 1}, y_{n + 1}\right)h\gamma_{2},
\end{array}\right.
\label{trbdf2deter}
\end{equation}
\noindent where $t_{n + \gamma} = t_{n} + \gamma h$, $\gamma \in \left(0, 1\right)$ is a parameter, $\gamma_{2} = \frac{1 - \gamma}{2 - \gamma}$ and $\gamma_{3} = \frac{1}{\gamma\left(2 - \gamma\right)}$.

The TR-BDF2 procedure is well-known and used because of its several interesting accuracy and stability properties. Especially important for stiff problems are the stability properties: $A$-stability and $L$-stability. These concepts can be found in detail in \cite{Hairer2010}. It is $A$-stable for any value of $\gamma$ and $L$-stable for $\gamma = 2 - \sqrt{2}$, which is shown in \cite{Bank1985}. From now on we will use this parameter value.

A first idea in order to solve stochastic equations is to use this method in the deterministic part and the usual approximation for the stochastic terms. However, this combination provides a method at most of order 1, as we will show later in Remark~\ref{remark2}. Consequently, the order 2 of the original method would be lost.

Our aim is to develop a method for stochastic differential equation problems, which generalizes the deterministic TR-BDF2, i.e., if the noise is zero, the same method is recovered and also has order 2 in the sense described in Definition~\ref{def}.

In order to obtain the method, the development of It{\^o}--Taylor of different orders, introduced in the last section, applied to some functions will be fundamental.

First of all, we will fix some notations that will be used to obtain the method:

\begin{itemize}
\item We denote by $\hat{\Delta}_{2N}$ a partition of the interval $\left[0, T\right]$ with $2N + 1$ points,
\begin{equation*}
\hat{\Delta}_{2N} = \left\{\hat{t}_{0} = 0 < \hat{t}_{1} < \hat{t}_{2} < \cdots < \hat{t}_{2N - 1} < \hat{t}_{2N} = T\right\},
\end{equation*}
\noindent where $\hat{t}_{2j} = t_{0} + jh$, $\forall j = 0, \ldots, N$ and $\hat{t}_{2j + 1} = \hat{t}_{2j} + \gamma h$, $\forall j = 0, \ldots, N - 1$. Note that the even indices of the partition correspond to the $\Delta_{N}$ partition defined in \eqref{Dn}.

\item $\Delta W_{s, t}^{j} \coloneqq I_{\left(j\right), s, t}$, the stochastic integral of one index, which satisfies $\Delta W_{s, t}^{j} \sim \mathcal{N}\left(0, t - s\right)$. In particular, $\forall k = 0, \ldots, N - 1$ and $n = 2k$:

\begin{itemize}
\item[*] $\Delta W_{\hat{t}_{2k}, \hat{t}_{2k + 1}}^{j} = W_{t_{n + \gamma}}^{j} - W_{t_{n}}^{j}$, with $\Delta W_{t_{n}, t_{n + \gamma}}^{j} \sim \mathcal{N}\left(0, h\gamma\right)$.

\item[*] $\Delta W_{\hat{t}_{2k + 1}, \hat{t}_{2k + 2}}^{j} = W_{t_{n + 1}}^{j} - W_{t_{n + \gamma}}^{j}$, with $\Delta W_{t_{n + \gamma}, t_{n + 1}}^{j} \sim \mathcal{N}\left(0, h\left(1 - \gamma\right)\right)$.
\end{itemize}

\item $a_{t} \coloneqq a\left(t, X_{t}\right)$, $L^{0}a_{t} \coloneqq L^{0}a\left(t, X_{t}\right)$, $L^{j}a_{t} \coloneqq L^{j}a\left(t, X_{t}\right)$, $j = 1, \ldots, m$ and for $j_{1}, j_{2} = 1, \ldots, m$, $L^{j_{1}}L^{j_{2}}a_{t} \coloneqq L^{j_{1}}L^{j_{2}}a\left(t, X_{t}\right)$, the coefficient functions corresponding to the deterministic part of the SDE, evaluated at the solution $X_{t}$ at time $t$.

\item $\hat{a}_{k} \coloneqq a\left(\hat{t}_{k}, Y_{k}\right)$, $L^{0}\hat{a}_{k} \coloneqq L^{0}a\left(\hat{t}_{k}, Y_{k}\right)$, $L^{j}\hat{a}_{k} \coloneqq L^{j}a\left(\hat{t}_{k}, Y_{k}\right)$, $j = 1, \ldots, m$ and for $j_{1}, j_{2} = 1, \ldots, m$, $L^{j_{1}}L^{j_{2}}\hat{a}_{k} \coloneqq L^{j_{1}}L^{j_{2}}a\left(\hat{t}_{k}, Y_{k}\right)$, the coefficient functions corresponding to the deterministic part of the SDE, evaluated at the approximation $Y_{k}$ at time $\hat{t}_{k} \in \hat{\Delta}_{2N}$.

\item In some parts of this paper, we denote by $t_{n} = \hat{t}_{2j}$ and $t_{n + \gamma} = \hat{t}_{2j + 1}$.
\end{itemize}

\subsection{Construction of the method}\label{sec41}
We start with the It{\^o}--Taylor expansion defined in \eqref{Ito-Taylor} of solution $X_{t}$ of order 2 from $t_{n}$ to $t_{n + \gamma}$:
\begin{equation}
X_{t_{n + \gamma}} = X_{t_{n}} + a_{t_{n}}h\gamma + \frac{1}{2}L^{0}a_{t_{n}}h^{2}\gamma^{2} + S_{2}\left[X_{t_{n}}\right]_{t_{n}, t_{n + \gamma}} + \tilde{R}_{2}\left[F\right]_{t_{n}, t_{n + \gamma}},
\label{aa}
\end{equation}
\noindent where we recall that ${}_{\alpha}F\left(t, x\right)$ are the coefficient functions of $F\left(t, x\right) \equiv x$ and $S_{2}\left[X_{t_{n}}\right]_{t_{n}, t_{n + \gamma}}$ corresponds to the stochastic terms.

For function $a$, we take the It{\^o}--Taylor expansion of order 1 from $t_{n}$ to $t_{n + \gamma}$:
\begin{equation}
a_{t_{n + \gamma}} = a_{t_{n}} + L^{0}a_{t_{n}}h\gamma + \sum_{j = 1}^{m}L^{j}a_{t_{n}}\Delta W_{t_{n}, t_{n + \gamma}}^{j} + \sum_{j_{1}, j_{2} = 1}^{m}L^{j_{1}}L^{j_{2}}a_{t_{n}}I_{\left(j_{1}, j_{2}\right), t_{n}, t_{n + \gamma}} + \tilde{R}_{1}\left[a\right]_{t_{n}, t_{n + \gamma}},
\label{aaaaa}
\end{equation}

Writing $a_{t_{n}}h\gamma = \frac{1}{2}a_{t_{n}}h\gamma + \frac{1}{2}a_{t_{n}}h\gamma$ in equality \eqref{aa}, we obtain:
\begin{equation}
X_{t_{n + \gamma}} = X_{t_{n}} + \frac{1}{2}\left(a_{t_{n}} + a_{t_{n + \gamma}}\right)h\gamma + P_{t_{n}} + R_{n},
\label{Xng}
\end{equation}
\noindent with
\begin{equation*}
P_{t_{n}} = S_{2}\left[X_{t_{n}}\right]_{t_{n}, t_{n + \gamma}} - \frac{1}{2}\sum_{j = 1}^{m}L^{j}a_{t_{n}}\Delta W_{t_{n}, t_{n + \gamma}}^{j}h\gamma - \frac{1}{2}\sum_{j_{1}, j_{2} = 1}^{m}L^{j_{1}}L^{j_{2}}a_{t_{n}}I_{\left(j_{1}, j_{2}\right), t_{n}, t_{n + \gamma}}h\gamma,
\end{equation*}
\noindent and
\begin{equation*}
R_{n} = \tilde{R}_{2}\left[F\right]_{t_{n}, t_{n + \gamma}} - \frac{1}{2}\tilde{R}_{1}\left[a\right]_{t_{n}, t_{n + \gamma}}h\gamma.
\end{equation*}

Now, let us consider again the It{\^o}--Taylor expansion \eqref{Ito-Taylor} of the solution $X_{t}$ of order 2 from $t_{n + \gamma}$ to $t_{n + 1}$
\begin{equation}
X_{t_{n + 1}} = X_{t_{n + \gamma}} + a_{t_{n + \gamma}}h\left(1 - \gamma\right) + \frac{1}{2}L^{0}a_{t_{n + \gamma}}h^{2}\left(1 - \gamma\right)^{2} + S_{2}\left[X_{t_{n + \gamma}}\right]_{t_{n + \gamma}, t_{n + 1}} + \tilde{R}_{2}\left[F\right]_{t_{n + \gamma}, t_{n + 1}}.
\label{otro}
\end{equation}

Taking $X_{t_{n + \gamma}} = \gamma_{3}X_{t_{n + \gamma}} + \left(1 - \gamma_{3}\right)X_{t_{n + \gamma}}$ in equation \eqref{otro}, considering the values of method defined in \eqref{trbdf2deter}, and replacing this equality in \eqref{Xng}, we obtain:
\begin{equation}
\begin{array}{l}
\displaystyle X_{t_{n + 1}} = \gamma_{3}X_{t_{n + \gamma}} + \left(1 - \gamma_{3}\right)\left(X_{t_{n}} + \frac{1}{2}\left(a_{t_{n}} + a_{t_{n + \gamma}}\right)h\gamma + P_{t_{n}} + R_{n}\right) + a_{t_{n + \gamma}}h\left(1 - \gamma\right)\\
\displaystyle \mathrel{\phantom{X_{t_{n + 1}} = }}+ \frac{1}{2}L^{0}a_{t_{n + \gamma}}h^{2}\left(1 - \gamma\right)^{2} + S_{2}\left[X_{t_{n + \gamma}}\right]_{t_{n + \gamma}, t_{n + 1}} + \tilde{R}_{2}\left[F\right]_{t_{n + \gamma}, t_{n + 1}}.
\end{array}
\label{n1}
\end{equation}

Using the same procedure as in the derivation of equality \eqref{aaaaa} between $t_{n + \gamma}$ and $t_{n + 1}$ allows us
\begin{equation*}
\begin{array}{l}
\displaystyle a_{t_{n + \gamma}} = a_{t_{n + 1}} - L^{0}a_{t_{n + \gamma}}h\left(1 - \gamma\right) - \sum_{j = 1}^{m}L^{j}a_{t_{n + \gamma}}\Delta W_{t_{n + \gamma}, t_{n + 1}}^{j} - \sum_{j_{1}, j_{2} = 1}^{m}L^{j_{1}}L^{j_{2}}a_{t_{n + \gamma}}I_{\left(j_{1}, j_{2}\right), t_{n + \gamma}, t_{n + 1}}\\
\displaystyle \mathrel{\phantom{a_{t_{n + \gamma}} = }}- \tilde{R}_{1}\left[a\right]_{t_{n + \gamma}, t_{n + 1}}.
\end{array}
\end{equation*}

On the other hand, considering now the It{\^o}--Taylor expansion \eqref{Ito-Taylor} of order 0 of $L^{0}a$ from $t_{n}$ to $t_{n + \gamma}$:
\begin{equation*}
L^{0}a_{t_{n}} = L^{0}a_{t_{n + \gamma}} - \tilde{R}_{0}\left[L^{0}a\right]_{t_{n}, t_{n + \gamma}},
\end{equation*}
\noindent from where we can write equation \eqref{n1} as follows:
\begin{equation*}
X_{t_{n + 1}} = \gamma_{3}X_{t_{n + \gamma}} + \left(1 - \gamma_{3}\right)X_{t_{n}} + a_{t_{n + 1}}h\gamma_{2} + Q_{t_{n}} + Q_{t_{n + \gamma}} + R_{n + \gamma}.
\end{equation*}

The stochastic terms are
\begin{equation*}
\begin{array}{l}
\displaystyle Q_{t_{n}} = S_{2}\left[X_{t_{n}}\right]_{t_{n}, t_{n + \gamma}}\left(1 - \gamma_{3}\right) - \sum_{j = 1}^{m}L^{j}a_{t_{n}}\Delta W_{t_{n}, t_{n + \gamma}}^{j}h\left(1 - \gamma_{3}\right)\gamma\\
\displaystyle \mathrel{\phantom{Q_{t_{n}} = }}- \sum_{j_{1}, j_{2} = 1}^{m}L^{j_{1}}L^{j_{2}}a_{t_{n}}I_{\left(j_{1}, j_{2}\right), t_{n}, t_{n + \gamma}}h\left(1 - \gamma_{3}\right)\gamma,
\end{array}
\end{equation*}
\begin{equation*}
Q_{t_{n + \gamma}} = S_{2}\left[X_{t_{n + \gamma}}\right]_{t_{n + \gamma}, t_{n + 1}} - \sum_{j = 1}^{m}L^{j}a_{t_{n + \gamma}}\Delta W_{t_{n + \gamma}, t_{n + 1}}^{j}h\gamma_{2} - \sum_{j_{1}, j_{2} = 1}^{m}L^{j_{1}}L^{j_{2}}a_{t_{n + \gamma}}I_{\left(j_{1}, j_{2}\right), t_{n + \gamma}, t_{n + 1}}h\gamma_{2},
\end{equation*}
\noindent and the remainder is
\begin{equation*}
\begin{array}{l}
\displaystyle R_{n + \gamma} = \tilde{R}_{2}\left[F\right]_{t_{n}, t_{n + \gamma}}\left(1 - \gamma_{3}\right) + \tilde{R}_{2}\left[F\right]_{t_{n + \gamma}, t_{n + 1}} - \tilde{R}_{1}\left[a\right]_{t_{n}, t_{n + \gamma}}h\left(1 - \gamma_{3}\right)\gamma - \tilde{R}_{1}\left[a\right]_{t_{n + \gamma}, t_{n + 1}}h\gamma_{2}\\
\displaystyle \mathrel{\phantom{R_{n + \gamma} = }}+ \frac{1}{2}\tilde{R}_{0}\left[L^{0}a\right]_{t_{n}, t_{n + \gamma}}h^{2}\left(1 - \gamma_{3}\right)\gamma^{2}.
\end{array}
\end{equation*}

Finally, the expansion of solution $X_{t}$ that we have obtained is:
\begin{equation}
\left\{\begin{array}{l}
\displaystyle X_{t_{n + \gamma}} = X_{t_{n}} + \frac{1}{2}\left(a_{t_{n}} + a_{t_{n + \gamma}}\right)h\gamma + P_{t_{n}} + R_{n},\\
\displaystyle X_{t_{n + 1}} = \gamma_{3}X_{t_{n + \gamma}} + \left(1 - \gamma_{3}\right)X_{t_{n}} + a_{t_{n + 1}}h\gamma_{2} + Q_{t_{n}} + Q_{t_{n + \gamma}} + R_{n + \gamma}.
\end{array}\right.
\label{eqex}
\end{equation}

We can conclude that the method we have constructed, which we will call stochastic TR-BDF2, is as follows
\begin{equation}
\left\{\begin{array}{l}
\displaystyle \text{Given } Y_{0}, \text{ compute } \forall n = 0, \ldots, N - 1,\\
\displaystyle Y_{n + \gamma} = Y_{n} + \frac{1}{2}\left(\hat{a}_{n} + \hat{a}_{n + \gamma}\right)h\gamma + \hat{P}_{n},\\
\displaystyle Y_{n + 1} = \gamma_{3}Y_{n + \gamma} + \left(1 - \gamma_{3}\right)Y_{n} + \hat{a}_{n + 1}h\gamma_{2} + \hat{Q}_{n} + \hat{Q}_{n + \gamma},
\end{array}\right.
\label{trbdf2}
\end{equation}
\noindent where, if we denote $\sum_{\alpha \in \mathcal{A}_{2} : l\left(\alpha\right) \neq n\left(\alpha\right)}{}_{\alpha}F\left(t_{n}, Y_{n}\right)I_{\alpha, t_{n}, t_{n + \gamma}}$ by $S_{2}\left[Y_{n}\right]$ and similarly, in the auxiliary step, we denote $\sum_{\alpha \in \mathcal{A}_{2} : l\left(\alpha\right) \neq n\left(\alpha\right)}{}_{\alpha}F\left(t_{n + \gamma}, Y_{n + \gamma}\right)I_{\alpha, t_{n + \gamma}, t_{n + 1}}$ by $S_{2}\left[Y_{n + \gamma}\right]$:
\begin{equation*}
\hat{P}_{n} = S_{2}\left[Y_{n}\right] - \frac{1}{2}\sum_{j = 1}^{m}L^{j}\hat{a}_{n}\Delta W_{t_{n}, t_{n + \gamma}}^{j}h\gamma - \frac{1}{2}\sum_{j_{1}, j_{2} = 1}^{m}L^{j_{1}}L^{j_{2}}\hat{a}_{n}I_{\left(j_{1}, j_{2}\right), t_{n}, t_{n + \gamma}}h\gamma,
\end{equation*}
\begin{equation*}
\displaystyle \hat{Q}_{n} = \left(S_{2}\left[Y_{n}\right] - \sum_{j = 1}^{m}\left(L^{j}\hat{a}_{n}\Delta W_{t_{n}, t_{n + \gamma}}^{j} + \sum_{j_{2} = 1}^{m}L^{j}L^{j_{2}}\hat{a}_{n}I_{\left(j, j_{2}\right), t_{n}, t_{n + \gamma}}\right)h\gamma\right)\left(1 - \gamma_{3}\right),
\end{equation*}
\begin{equation*}
\displaystyle \hat{Q}_{n + \gamma} = S_{2}\left[Y_{n + \gamma}\right] - \sum_{j = 1}^{m}\left(L^{j}\hat{a}_{n + \gamma}\Delta W_{t_{n + \gamma}, t_{n + 1}}^{j} + \sum_{j_{2} = 1}^{m}L^{j}L^{j_{2}}\hat{a}_{n + \gamma}I_{\left(j, j_{2}\right), t_{n + \gamma}, t_{n + 1}}\right)h\gamma_{2}.
\end{equation*}

\begin{remark}
Note that when the stochastic part is zero, i.e., if $b^{j} \equiv 0$ for all $j = 1, \ldots, m$, we recover the TR-BDF2 method as proposed at the beginning of this section.
\end{remark}

\begin{remark}\label{remark2}
We could think of simplifying the expressions of the stochastic part of the method by suppressing some of the summations that appear, since they are still remnants of the different It{\^o}--Taylor expansions considered, but this leads to the loss of order 2 of the method, as we will see in the section on numerical validation. We will now define different methods by suppressing some of these terms, which we will denote by TR-BDF2$^{r}$, with $r = 1, \ldots, 6$ and we will study their approximate order of convergence.

All these methods are obtained from scheme \eqref{trbdf2} with different definitions of $\hat{P}_{n}$, $\hat{Q}_{n}$ and $\hat{Q}_{n + \gamma}$:

\begin{itemize}
\item TR-BDF2$^{1}$
\begin{equation*}
\hat{P}_{n} = S_{2}\left[Y_{n}\right] - \frac{1}{2}\sum_{j, j_{2} = 1}^{m}L^{j}L^{j_{2}}\hat{a}_{n}I_{\left(j, j_{2}\right), t_{n}, t_{n + \gamma}}h\gamma.
\end{equation*}

\item TR-BDF2$^{2}$
\begin{equation*}
\hat{P}_{n} = S_{2}\left[Y_{n}\right] - \frac{1}{2}\sum_{j = 1}^{m}L^{j}\hat{a}_{n}\Delta W_{t_{n}, t_{n + \gamma}}^{j}h\gamma.
\end{equation*}

\item TR-BDF2$^{3}$
\begin{equation*}
\hat{Q}_{n} = S_{2}\left[Y_{n}\right]\left(1 - \gamma_{3}\right) - \sum_{j, j_{2} = 1}^{m}L^{j}L^{j_{2}}\hat{a}_{n}I_{\left(j, j_{2}\right), t_{n}, t_{n + \gamma}}h\left(1 - \gamma_{3}\right)\gamma.
\end{equation*}

\item TR-BDF2$^{4}$
\begin{equation*}
\hat{Q}_{n} = S_{2}\left[Y_{n}\right]\left(1 - \gamma_{3}\right) - \sum_{j = 1}^{m}L^{j}\hat{a}_{n}\Delta W_{t_{n}, t_{n + \gamma}}^{j}h\left(1 - \gamma_{3}\right)\gamma.
\end{equation*}

\item TR-BDF2$^{5}$
\begin{equation*}
\hat{Q}_{n + \gamma} = S_{2}\left[Y_{n + \gamma}\right] - \sum_{j, j_{2} = 1}^{m}L^{j}L^{j_{2}}\hat{a}_{n + \gamma}I_{\left(j, j_{2}\right), t_{n + \gamma}, t_{n + 1}}h\gamma_{2}.
\end{equation*}

\item TR-BDF2$^{6}$
\begin{equation*}
\hat{Q}_{n + \gamma} = S_{2}\left[Y_{n + \gamma}\right] - \sum_{j = 1}^{m}L^{j}\hat{a}_{n + \gamma}\Delta W_{t_{n + \gamma}, t_{n + 1}}^{j}h\gamma_{2}.
\end{equation*}
\end{itemize}
\end{remark}

\subsection{Convergence order}
To prove that the developed scheme \eqref{trbdf2} is indeed of order 2 we need some preliminary results, following an argument similar to that of Ewald et al. in \cite{Ewald2005}, where the theorems developed cannot be applied directly due to the structure of the method, although the techniques employed are analogous. These results are proved for uniform partitions but are easily extensible to non-uniform partitions which will be the ones we will use, since our $\hat{\Delta}_{2N}$ partition is not uniform. Throughout of this paper, $C$ is a positive constant whose value may change from line to line.

\begin{proposition}\label{bueno1}
Suppose $Y_{n}$ is a stochastic process adapted to the filtration $\mathcal{F}_{t}$ at the equipartition (i.e., $Y_{n}$ is $\mathcal{F}_{t_{n}}$-measurable) and it is also well-defined in the auxiliary steps $t_{n + \gamma}$, with $E\left[\sup_{0 \leq n \leq N}\left\|Y_{n}\right\|^{2}\right] < \infty$. Also, function $f$ belongs to ${\rm C}^{1, 2}$, and $\alpha$ is a multi-index with $l\left(\alpha\right) \geq 1$. Then, there exists a constant $C > 0$ such that
\begin{equation*}
E\left[\sup_{0 \leq l \leq n}\left\|\sum_{k = 0}^{l - 1}I_{\alpha}\left[f\left(t_{k}, X_{t_{k}}\right) - f\left(t_{k}, Y_{k}\right)\right]_{t_{k}, t_{k + \gamma}}\right\|^{2}\right] \leq Ch\sum_{k = 0}^{n - 1}E\left[\sup_{0 \leq l \leq k}\left\|X_{t_{l}} - Y_{l}\right\|^{2}\right].
\end{equation*}
\begin{equation*}
E\left[\sup_{0 \leq l \leq n}\left\|\sum_{k = 0}^{l - 1}I_{\alpha}\left[f\left(t_{k + \gamma}, X_{t_{k + \gamma}}\right) - f\left(t_{k + \gamma}, Y_{k + \gamma}\right)\right]_{t_{k + \gamma}, t_{k + 1}}\right\|^{2}\right] \leq Ch\sum_{k = 0}^{n - 1}E\left[\sup_{0 \leq l \leq k}\left\|X_{t_{l + \gamma}} - Y_{l + \gamma}\right\|^{2}\right].
\end{equation*}
\end{proposition}

\begin{remark}
The proof of the Proposition above is analogous to that of the corresponding result in \cite{Ewald2005}, with the difference being the replacement of the global Lipschitz condition by the mean value theorem, which is applicable in this case due to the continuity assumption and the boundedness of the second moment of the solution and the approximation.
\end{remark}

\begin{proposition}\label{bueno2}
Suppose the function $f$ satisfies the linear growth bound condition, and $\alpha$ is a multi-index with $l\left(\alpha\right) \geq 1$. Then,
\begin{equation*}
E\left[\sup_{0 \leq l \leq n}\left\|\sum_{k = 0}^{l - 1}I_{\alpha}\left[f\left(\cdot, X_{\cdot}\right)\right]_{t_{k}, t_{k + \gamma}}\right\|^{2}\right] \leq Ch^{\Phi\left(\alpha\right)}\left(1 + E\left[\left\|X_{0}\right\|^{2}\right]\right),
\end{equation*}
\begin{equation*}
E\left[\sup_{0 \leq l \leq n}\left\|\sum_{k = 0}^{l - 1}I_{\alpha}\left[f\left(\cdot, X_{\cdot}\right)\right]_{t_{k + \gamma}, t_{k + 1}}\right\|^{2}\right] \leq Ch^{\Phi\left(\alpha\right)}\left(1 + E\left[\left\|X_{0}\right\|^{2}\right]\right),
\end{equation*}
\noindent where
\begin{equation*}
\Phi\left(\alpha\right) = \left\{\begin{array}{ll}
\displaystyle 2l\left(\alpha\right) - 2, & \displaystyle l\left(\alpha\right) = n\left(\alpha\right),\\
\displaystyle l\left(\alpha\right) + n\left(\alpha\right) - 1, & \displaystyle l\left(\alpha\right) \neq n\left(\alpha\right).
\end{array}\right.
\end{equation*}
\end{proposition}

\begin{proposition}\label{bueno3}
Suppose that a sequence of positive numbers $Z_{n}$, for $n = 0, 1, \ldots, N$, and $Z_{n + \gamma}$ for $n = 0, 1, \ldots, N - 1$ with $\gamma \in \left(0, 1\right)$, satisfies the inequality
\begin{equation*}
Z_{n} \leq C\left(h\sum_{k = 0}^{n - 1}Z_{k} + h\sum_{k = 0}^{n - 1}Z_{k + \gamma} + h^{p}\right),
\end{equation*}
\noindent for some positive constant $C$ and some $p > 0$. Then, $Z_{N} = \mathcal{O}\left(h^{p}\right)$ as $h \to 0$.
\end{proposition}

To conclude with the previous results, we state the following Lemma reproduced from \cite{Buckwar2021}, useful for the implicitness of the scheme:

\begin{lemma}\label{lem1}
Under Hypotheses~\ref{H} with $p = 2$, given $y_{1}, y_{2} \in \mathbb{R}^{d}$ and $h \in \left(0, \frac{1}{2\beta}\right)$, with $\beta = \max\left(C + \frac{1}{2}, C_{1}\right)$, let $x_{1}, x_{2} \in \mathbb{R}^{n}$ satisfy the implicit equation
\begin{equation*}
x_{i} - ha\left(t, x_{i}\right) = y_{i}, i = 1, 2.
\end{equation*}

Then, $x_{1}$ and $x_{2}$ exist, are unique and satisfy the inequality
\begin{equation*}
\left(1 - 2hC\right)\left\|x_{1} - x_{2}\right\|^{2} \leq \left\|y_{1} - y_{2}\right\|^{2}.
\end{equation*}
\end{lemma}

\begin{remark}
Since $h$ is the discretization parameter, by definition it tends to zero. This implies that the condition of the previous lemma, $h \in \left(0, \frac{1}{2\beta}\right)$, is not restrictive, as we are only concerned with small values of $h$. In fact, the condition of the previous lemma can be equivalently rewritten as
\begin{equation}\label{lemmaeq}
\left\|x_{1} - x_{2}\right\|^{2} \leq K\left\|y_{1} - y_{2}\right\|^{2}.
\end{equation}
\end{remark}

Before proving that our scheme has order 2, we show that the numerical approximation \eqref{trbdf2} has bounded second moment.

\begin{proposition}
Suppose that Hypotheses~\ref{H} are fulfilled for $p = 2$ and
\begin{equation}
E\left[\left\|Y_{0}\right\|^{2}\right] < \infty.
\label{inicial}
\end{equation}

Then,
\begin{equation*}
E\left[\sup_{0 \leq n \leq N}\left\|Y_{n}\right\|^{2}\right] < \infty.
\end{equation*}
\end{proposition}

\begin{proof}
From \eqref{trbdf2} we have
\begin{equation*}
Y_{k + 1} = Y_{k} + \frac{1}{2}\left(\hat{a}_{k} + \hat{a}_{k + \gamma}\right)h\gamma_{3}\gamma + \gamma_{3}\hat{P}_{k} + \hat{a}_{k + 1}h\gamma_{2} + \hat{Q}_{k} + \hat{Q}_{k + \gamma}.
\end{equation*}

\noindent Then, summing
\begin{equation*}
\sum_{k = 0}^{n - 1}Y_{k + 1} = \sum_{k = 0}^{n - 1}\left(Y_{k} + \frac{1}{2}\left(\hat{a}_{k} + \hat{a}_{k + \gamma}\right)h\gamma_{3}\gamma + \gamma_{3}\hat{P}_{k} + \hat{a}_{k + 1}h\gamma_{2} + \hat{Q}_{k} + \hat{Q}_{k + \gamma}\right).
\end{equation*}

\noindent We can rewrite it as
\begin{equation*}
Y_{n} = Y_{0} + \frac{1}{2}\sum_{k = 0}^{n - 1}\left(\left(\hat{a}_{k} + \hat{a}_{k + \gamma}\right)h\gamma_{3}\gamma + 2\hat{a}_{k + 1}h\gamma_{2}\right) + \sum_{k = 0}^{n - 1}\left(\gamma_{3}\hat{P}_{k} + \hat{Q}_{k} + \hat{Q}_{k + \gamma}\right).
\end{equation*}

\noindent Expanding the first summation, we obtain
\begin{equation*}
\sum_{k = 0}^{n - 1}\left(\left(\hat{a}_{k} + \hat{a}_{k + \gamma}\right)h\gamma_{3}\gamma + 2\hat{a}_{k + 1}h\gamma_{2}\right) = \sum_{k = 0}^{n - 1}\left(\hat{a}_{k}A_{\gamma}^{k} + \hat{a}_{k + \gamma}h\gamma_{3}\gamma\right) + 2\hat{a}_{n}h\gamma_{2},
\end{equation*}
\noindent with
\begin{equation*}
A_{\gamma}^{k} = \left\{\begin{array}{ll}
\displaystyle \gamma_{3}\gamma, & \displaystyle k = 0,\\
\displaystyle \gamma_{3}\gamma + 2\gamma_{2}, & \displaystyle k = 1, \ldots, n - 1.
\end{array}\right.
\end{equation*}

\noindent We rewrite the expression of $Y_{n}$ as
\begin{equation*}
Y_{n} = Y_{0} + \sum_{k = 0}^{n - 1}\left(\frac{1}{2}\hat{a}_{k}A_{\gamma}^{k}h + \frac{1}{2}\hat{a}_{k + \gamma}h\gamma_{3}\gamma + \gamma_{3}\hat{P}_{k} + \hat{Q}_{k} + \hat{Q}_{k + \gamma}\right) + \hat{a}_{n}h\gamma_{2}.
\end{equation*}

\noindent Taking inner product:
\begin{equation*}
\left\|Y_{n}\right\|^{2} = \left\langle Y_{0} + \sum_{k = 0}^{n - 1}\left(\frac{1}{2}\hat{a}_{k}A_{\gamma}^{k}h + \frac{1}{2}\hat{a}_{k + \gamma}h\gamma_{3}\gamma + \gamma_{3}\hat{P}_{k} + \hat{Q}_{k} + \hat{Q}_{k + \gamma}\right), Y_{n}\right\rangle + \left\langle\hat{a}_{n}h\gamma_{2}, Y_{n}\right\rangle.
\end{equation*}

\noindent Using \eqref{glboun}
\begin{equation*}
\begin{array}{l}
\displaystyle \left\|Y_{n}\right\|^{2} \leq \frac{1}{2}\left\|Y_{0} + \sum_{k = 0}^{n - 1}\left(\frac{1}{2}\hat{a}_{k}A_{\gamma}^{k}h + \frac{1}{2}\hat{a}_{k + \gamma}h\gamma_{3}\gamma + \gamma_{3}\hat{P}_{k} + \hat{Q}_{k} + \hat{Q}_{k + \gamma}\right)\right\|^{2} + \frac{1}{2}\left\|Y_{n}\right\|^{2}\\
\displaystyle \mathrel{\phantom{\left\|Y_{n}\right\|^{2} \leq }}+ h\gamma_{2}K\left(1 + \left\|Y_{n}\right\|^{2}\right).
\end{array}
\end{equation*}

\noindent Thus, using the notation $K^{\ast} = \left(1 - 2h\gamma_{2}K\right)^{- 1}$ and let $C$ be an arbitrary constant that may change in each inequality
\begin{equation*}
\left\|Y_{n}\right\|^{2} \leq K^{\ast}\left(\left\|Y_{0} + \sum_{k = 0}^{n - 1}\left(\frac{1}{2}\hat{a}_{k}A_{\gamma}^{k}h + \frac{1}{2}\hat{a}_{k + \gamma}h\gamma_{3}\gamma + \gamma_{3}\hat{P}_{k} + \hat{Q}_{k} + \hat{Q}_{k + \gamma}\right)\right\|^{2} + hC\right).
\end{equation*}

\noindent Expanding the square:
\begin{equation}
\left\|Y_{n}\right\|^{2} \leq C\left(\left\|Y_{0}\right\|^{2} + \left\|\sum_{k = 0}^{n - 1}\hat{a}_{k}h\right\|^{2} + \left\|\sum_{k = 0}^{n - 1}\hat{a}_{k + \gamma}h\right\|^{2} + \left\|\sum_{k = 0}^{n - 1}\hat{P}_{k}\right\|^{2} + \left\|\sum_{k = 0}^{n - 1}\hat{Q}_{k}\right\|^{2} + \left\|\sum_{k = 0}^{n - 1}\hat{Q}_{k + \gamma}\right\|^{2}\right).
\label{yn}
\end{equation}

\noindent Now, we proceed like the proof of Proposition~\ref{bueno1}. Firstly, note that for the deterministic terms the following is satisfied:
\begin{equation*}
\left\|\sum_{k = 0}^{n - 1}\hat{a}_{k}h\right\|^{2} \leq h^{2}N\sum_{k = 0}^{n - 1}\left\|\hat{a}_{k}\right\|^{2} = Th\sum_{k = 0}^{n - 1}\left\|\hat{a}_{k}\right\|^{2} \leq C + Ch\sum_{k = 0}^{n - 1}\left\|Y_{k}\right\|^{2},
\end{equation*}
\begin{equation*}
\left\|\sum_{k = 0}^{n - 1}\hat{a}_{k + \gamma}h\right\|^{2} \leq h^{2}N\sum_{k = 0}^{n - 1}\left\|\hat{a}_{k + \gamma}\right\|^{2} = Th\sum_{k = 0}^{n - 1}\left\|\hat{a}_{k + \gamma}\right\|^{2} \leq C + Ch\sum_{k = 0}^{n - 1}\left\|Y_{k + \gamma}\right\|^{2},
\end{equation*}
\noindent where we have used the growth bound condition in the last inequality.

\noindent On the other hand, for the stochastic terms we use the independent increments property of the stochastic integral to obtain
\begin{equation*}
E\left[\left\|\sum_{k = 0}^{n - 1}\hat{P}_{k}\right\|^{2} + \left\|\sum_{k = 0}^{n - 1}\hat{Q}_{k}\right\|^{2} + \left\|\sum_{k = 0}^{n - 1}\hat{Q}_{k + \gamma}\right\|^{2}\right] \leq C\sum_{k = 0}^{n - 1}E\left[\left\|\hat{P}_{k}\right\|^{2} + \left\|\hat{Q}_{k}\right\|^{2} + \left\|\hat{Q}_{k + \gamma}\right\|^{2}\right].
\end{equation*}

\noindent As mentioned earlier, just as in the proof of Proposition~\ref{bueno1}, it is easy to see that all the stochastic terms satisfy that they are at least of order $h$. Therefore, applying the growth bound condition we obtain
\begin{equation*}
C\sum_{k = 0}^{n - 1}E\left[\left\|\hat{P}_{k}\right\|^{2} + \left\|\hat{Q}_{k}\right\|^{2} + \left\|\hat{Q}_{k + \gamma}\right\|^{2}\right] \leq C + Ch\sum_{k = 0}^{n - 1}E\left[\left\|Y_{k}\right\|^{2}\right] + Ch\sum_{k = 0}^{n - 1}E\left[\left\|Y_{k + \gamma}\right\|^{2}\right].
\end{equation*}

\noindent Finally, taking expected value and supremum in \eqref{yn}, and using \eqref{inicial} and the above inequalities, we conclude
\begin{equation*}
E\left[\sup_{0 \leq l \leq n}\left\|Y_{l}\right\|^{2}\right] \leq C + Ch\sum_{k = 0}^{n - 1}E\left[\sup_{0 \leq l \leq k}\left\|Y_{l}\right\|^{2}\right] + Ch\sum_{k = 0}^{n - 1}E\left[\sup_{0 \leq l \leq k}\left\|Y_{l + \gamma}\right\|^{2}\right]
\end{equation*}

\noindent Using the discrete Gronwall inequality:
\begin{equation*}
E\left[\sup_{0 \leq l \leq N}\left\|Y_{l}\right\|^{2}\right] \leq Ce^{ChN} \leq Ce^{CT},
\end{equation*}
\noindent thus concluding the proof.
\end{proof}

Having presented the preliminary results, we are in a position to prove that the stochastic TR-BDF2 method is of order 2 in the sense of Definition~\ref{def}.

\begin{theorem}
Let $\left\{Y_{n}\right\}_{0 \leq n \leq N}$ be the stochastic TR-BDF2 scheme given by \eqref{trbdf2}. Suppose that Hypotheses~\ref{H} with (H.2$\ast$) is fulfilled for $p = 2$, the initial condition satisfies \eqref{inicial} and
\begin{equation*}
\left(E\left[\left\|X_{0} - Y_{0}\right\|^{2}\right]\right)^{\frac{1}{2}} \leq \mathcal{O}\left(h^{2}\right).
\end{equation*}

Then, $\left\{Y_{n}\right\}_{0 \leq n \leq N}$ converges to $X_{t}$ with order 2.
\end{theorem}

\begin{proof}
Let $e_{k} = X_{t_{k}} - Y_{k}$ for $k = 0, \ldots, N$, and $e_{k + \gamma} = X_{t_{k + \gamma}} - Y_{k + \gamma}$ for $k = 0, \ldots, N - 1$. Subtracting the corresponding part of \eqref{trbdf2} from equation \eqref{eqex}, we obtain
\begin{equation*}
e_{k + 1} = \gamma_{3}e_{k + \gamma} + \left(1 - \gamma_{3}\right)e_{k} + \left(a_{t_{k + 1}} - \hat{a}_{k + 1}\right)h\gamma_{2} + Q_{t_{k}} - \hat{Q}_{k} + Q_{t_{k + \gamma}} - \hat{Q}_{k + \gamma} + R_{k + \gamma}.
\end{equation*}

\noindent Analogously, taking \eqref{Xng} and \eqref{trbdf2} we obtain:
\begin{equation*}
e_{k + \gamma} = e_{k} + \frac{1}{2}\left(a_{t_{k}} - \hat{a}_{k}\right)h\gamma + \frac{1}{2}\left(a_{t_{k + \gamma}} - \hat{a}_{k + \gamma}\right)h\gamma + P_{t_{k}} - \hat{P}_{k} + R_{k}.
\end{equation*}

\noindent Then,
\begin{equation*}
\begin{array}{l}
\displaystyle e_{k + 1} = e_{k} + \frac{1}{2}\left(a_{t_{k}} - \hat{a}_{k}\right)h\gamma_{3}\gamma + \frac{1}{2}\left(a_{t_{k + \gamma}} - \hat{a}_{k + \gamma}\right)h\gamma_{3}\gamma + \gamma_{3}\left(P_{t_{k}} - \hat{P}_{k}\right) + \left(a_{t_{k + 1}} - \hat{a}_{k + 1}\right)h\gamma_{2}\\
\displaystyle \mathrel{\phantom{e_{k + 1} = }}+ Q_{t_{k}} - \hat{Q}_{k} + Q_{t_{k + \gamma}} - \hat{Q}_{k + \gamma} + \gamma_{3}R_{k} + R_{k + \gamma}.
\end{array}
\end{equation*}

\noindent Adding from $k = 0$ to $k = n$, we have:
\begin{equation*}
\begin{array}{l}
\displaystyle e_{n + 1} = e_{0} + \sum_{k = 0}^{n}\left(\frac{1}{2}\left(a_{t_{k}} - \hat{a}_{k}\right)h\gamma_{3}\gamma + \frac{1}{2}\left(a_{t_{k + \gamma}} - \hat{a}_{k + \gamma}\right)h\gamma_{3}\gamma + \gamma_{3}\left(P_{t_{k}} - \hat{P}_{k}\right)\right) + \sum_{k = 1}^{n + 1}\left(a_{t_{k}} - \hat{a}_{k}\right)h\gamma_{2}\\
\displaystyle \mathrel{\phantom{e_{n + 1} = }}+ \sum_{k = 0}^{n}\left(Q_{t_{k}} - \hat{Q}_{k} + Q_{t_{k + \gamma}} - \hat{Q}_{k + \gamma} + \gamma_{3}R_{k} + R_{k + \gamma}\right).
\end{array}
\end{equation*}

\noindent Using Lemma~\ref{lem1} and \eqref{lemmaeq}, we obtain
\begin{equation*}
\begin{array}{l}
\displaystyle \left\|e_{n + 1}\right\|^{2} \leq K\left\|e_{0} + \sum_{k = 0}^{n}\left(\left(a_{t_{k}} - \hat{a}_{k}\right)h + \left(a_{t_{k + \gamma}} - \hat{a}_{k + \gamma}\right)h + P_{t_{k}} - \hat{P}_{k} + Q_{t_{k}} - \hat{Q}_{k} + Q_{t_{k + \gamma}} - \hat{Q}_{k + \gamma}\right)\vphantom{ + \sum_{k = 0}^{n}\left(R_{k} + R_{k + \gamma}\right)}\right.\\
\displaystyle \mathrel{\phantom{e_{n + 1} = }}\left.\vphantom{e_{0} + \sum_{k = 0}^{n}\left(\left(a_{t_{k}} - \hat{a}_{k}\right)h + \left(a_{t_{k + \gamma}} - \hat{a}_{k + \gamma}\right)h + P_{t_{k}} - \hat{P}_{k} + Q_{t_{k}} - \hat{Q}_{k} + Q_{t_{k + \gamma}} - \hat{Q}_{k + \gamma}\right) }+ \sum_{k = 0}^{n}\left(R_{k} + R_{k + \gamma}\right)\right\|^{2}.
\end{array}
\end{equation*}

\noindent Set $Z_{n} = E\left[\sup_{0 \leq l \leq n}\left\|e_{l}\right\|^{2}\right]$ and $Z_{n + \gamma} = E\left[\sup_{0 \leq l \leq n}\left\|e_{l + \gamma}\right\|^{2}\right]$, we can deduce that
\begin{equation*}
\begin{array}{l}
\displaystyle Z_{n} \leq C\left(Z_{0} + E\left[\sup_{0 \leq l \leq n}\left(\left\|\sum_{k = 0}^{l - 1}\left(a_{t_{k}} - \hat{a}_{k}\right)h\right\|^{2} + \left\|\sum_{k = 0}^{l - 1}\left(a_{t_{k + \gamma}} - \hat{a}_{k + \gamma}\right)h\right\|^{2} + \left\|\sum_{k = 0}^{l - 1}\left(P_{t_{k}} - \hat{P}_{k}\right)\right\|^{2}\vphantom{ + \left\|\sum_{k = 0}^{l - 1}\left(Q_{t_{k}} - \hat{Q}_{k}\right)\right\|^{2} + \left\|\sum_{k = 0}^{l - 1}\left(Q_{t_{k + \gamma}} - \hat{Q}_{k + \gamma}\right)\right\|^{2} + \left\|\sum_{k = 0}^{l - 1}R_{k}\right\|^{2} + \left\|\sum_{k = 0}^{l - 1}R_{k + \gamma}\right\|^{2}}\right.\right.\right.\\
\displaystyle \mathrel{\phantom{Z_{n} \leq }}\left.\left.\left.\vphantom{Z_{0} + E\left[\sup_{0 \leq l \leq n}\left(\left\|\sum_{k = 0}^{l - 1}\left(a_{t_{k}} - \hat{a}_{k}\right)h\right\|^{2} + \left\|\sum_{k = 0}^{l - 1}\left(a_{t_{k + \gamma}} - \hat{a}_{k + \gamma}\right)h\right\|^{2} + \left\|\sum_{k = 0}^{l - 1}\left(P_{t_{k}} - \hat{P}_{k}\right)\right\|^{2}\right.\right. }+ \left\|\sum_{k = 0}^{l - 1}\left(Q_{t_{k}} - \hat{Q}_{k}\right)\right\|^{2} + \left\|\sum_{k = 0}^{l - 1}\left(Q_{t_{k + \gamma}} - \hat{Q}_{k + \gamma}\right)\right\|^{2} + \left\|\sum_{k = 0}^{l - 1}R_{k}\right\|^{2} + \left\|\sum_{k = 0}^{l - 1}R_{k + \gamma}\right\|^{2}\right)\right]\right).
\end{array}
\end{equation*}

\noindent We study now each term. Firstly, note that all functions that appear in the summations of the above inequality are coefficient functions ${}_{\alpha}F$, for some $\alpha \in \mathcal{A}_{2} \cup \mathcal{B}\left(\mathcal{A}_{2}\right)$, i.e., they verify Hypotheses~\ref{H}.

\noindent In particular, they satisfy the hypotheses of Proposition~\ref{bueno1}, obtaining:
\begin{equation*}
E\left[\sup_{0 \leq l \leq n}\left\|\sum_{k = 0}^{l - 1}\left(a_{t_{k}} - \hat{a}_{k}\right)h\right\|^{2}\right] \leq Ch\sum_{k = 0}^{n - 1}Z_{k},
\end{equation*}
\begin{equation*}
E\left[\sup_{0 \leq l \leq n}\left\|\sum_{k = 0}^{l - 1}\left(a_{t_{k + \gamma}} - \hat{a}_{k + \gamma}\right)h\right\|^{2}\right] \leq Ch\sum_{k = 0}^{n - 1}Z_{k + \gamma},
\end{equation*}
\begin{equation*}
E\left[\sup_{0 \leq l \leq n}\left\|\sum_{k = 0}^{l - 1}\left(P_{t_{k}} - \hat{P}_{k}\right)\right\|^{2}\right] \leq Ch\sum_{k = 0}^{n - 1}Z_{k},
\end{equation*}
\begin{equation*}
E\left[\sup_{0 \leq l \leq n}\left\|\sum_{k = 0}^{l - 1}\left(Q_{t_{k}} - \hat{Q}_{k}\right)\right\|^{2}\right] \leq Ch\sum_{k = 0}^{n - 1}Z_{k},
\end{equation*}
\begin{equation*}
E\left[\sup_{0 \leq l \leq n}\left\|\sum_{k = 0}^{l - 1}\left(Q_{t_{k + \gamma}} - \hat{Q}_{k + \gamma}\right)\right\|^{2}\right] \leq Ch\sum_{k = 0}^{n - 1}Z_{k + \gamma}.
\end{equation*}

\noindent On the other hand, let us study the remaining terms. Using Proposition~\ref{bueno2} in each summation of the remainder sets, it holds:
\begin{equation*}
E\left[\sup_{0 \leq l \leq n}\left\|\sum_{k = 0}^{l - 1}\tilde{R}_{0}\left[L^{0}a\right]_{t_{k}, t_{k + \gamma}}\right\|^{2}\right]h^{4} \leq Ch^{0}\left(1 + \left\|X_{0}\right\|^{2}\right)h^{4},
\end{equation*}
\begin{equation*}
E\left[\sup_{0 \leq l \leq n}\left\|\sum_{k = 0}^{l - 1}\tilde{R}_{1}\left[a\right]_{t_{k}, t_{k + \gamma}}\right\|^{2}\right]h^{2} \leq Ch^{2}\left(1 + \left\|X_{0}\right\|^{2}\right)h^{2},
\end{equation*}
\begin{equation*}
E\left[\sup_{0 \leq l \leq n}\left\|\sum_{k = 0}^{l - 1}\tilde{R}_{1}\left[a\right]_{t_{k + \gamma}, t_{k + 1}}\right\|^{2}\right]h^{2} \leq Ch^{2}\left(1 + \left\|X_{0}\right\|^{2}\right)h^{2},
\end{equation*}
\begin{equation*}
E\left[\sup_{0 \leq l \leq n}\left\|\sum_{k = 0}^{l - 1}\tilde{R}_{2}\left[F\right]_{t_{k}, t_{k + \gamma}}\right\|^{2}\right] \leq Ch^{4}\left(1 + \left\|X_{0}\right\|^{2}\right),
\end{equation*}
\begin{equation*}
E\left[\sup_{0 \leq l \leq n}\left\|\sum_{k = 0}^{l - 1}\tilde{R}_{2}\left[F\right]_{t_{k + \gamma}, t_{k + 1}}\right\|^{2}\right] \leq Ch^{4}\left(1 + \left\|X_{0}\right\|^{2}\right).
\end{equation*}

\noindent Overall, we obtain
\begin{equation*}
Z_{n} \leq C\left(h\sum_{k = 0}^{n - 1}Z_{k} + h\sum_{k = 0}^{n - 1}Z_{k + \gamma} + Z_{0} + h^{4}\left(1 + \left\|X_{0}\right\|^{2}\right)\right).
\end{equation*}

\noindent Finally, using the hypothesis about the initial condition, we obtain:
\begin{equation*}
Z_{n} \leq C\left(h\sum_{k = 0}^{n - 1}Z_{k} + h\sum_{k = 0}^{n - 1}Z_{k + \gamma} + \mathcal{O}\left(h^{4}\right)\right).
\end{equation*}

\noindent Applying Proposition~\ref{bueno3} and taking the square root, the result follows.
\end{proof}

\section{Numerical stability}\label{s5}
In this section we will study the stability property of the proposed method. It is well known that for stiff problems, for example those known as fast slow problems, a very small time step is necessary to solve the fast variables, but this is due to reasons of having a good error behavior. However, in these problems there is an additional condition, it is necessary to choose this time step not only for error behavior but also for stability reasons.

In the literature, two types of stability are commonly studied, both aim to generalize the stability analysis in the deterministic framework, where the complex-valued linear differential equations
\begin{equation*}
x' = \lambda x,
\end{equation*}
\noindent are studied, where $\lambda \in \mathbb{C}$, with $\Re\left(\lambda\right) < 0$.

The first extension corresponds to consider the class of complex-valued linear test equations with additive noise. This case has been studied in depth by Kloeden and Platen in \cite{Kloeden1992}. They take the following equation
\begin{equation}
{\rm d}X_{t} = \lambda X_{t}{\rm d}t + \sigma {\rm d}W_{t},
\label{test1}
\end{equation}
\noindent where $\lambda$ and $\sigma$ are complex numbers.

Using this equation, a first concept of numerical stability is defined: $A$-stability.

\begin{definition}
We consider a numerical scheme such that, applied to the test equation \eqref{test1} can be written in the recursive form as
\begin{equation}
Y_{n + 1} = G\left(z\right)Y_{n} + V_{n},
\label{recur}
\end{equation}
\noindent where any $V_{n}$ does not depend on any $Y_{n}$, $n = 0, 1, \ldots, N - 1$ and $G$ is a mapping of the complex plane into itself. We call the region of absolute stability of the scheme to the set of complex numbers
\begin{equation*}
\left\{z \in \mathbb{C} : \left|G\left(z\right)\right| < 1\right\}.
\end{equation*}

Then, a stochastic scheme is $A$-stable if its region of absolute stability is the left half of the complex plane.
\end{definition}

Note that this definition of $A$-stability is the same as in the deterministic case when the noise is zero.

A second extension proposed by Lord et al. \cite{Lord2014} consists of adding a multiplicative noise giving rise to another stability concept: $MS$-stability. In this case the test equation is:
\begin{equation}
{\rm d}X_{t} = \lambda X_{t}{\rm d}t + \sigma X_{t}{\rm d}W_{t}.
\label{test2}
\end{equation}

\begin{definition}
We consider a numerical scheme such that, applied to the test equation \eqref{test2} can be written in the recursive form as
\begin{equation*}
Y_{n + 1} = R\left(h, \lambda, \sigma, I_{\alpha}\right)Y_{n},
\end{equation*}
\noindent where $I_{\alpha}$ represents the integrals considered in that scheme. Then, a scheme is $MS$-stable if and only if
\begin{equation*}
E\left[\left|R\left(h, \lambda, \sigma, I_{\alpha}\right)\right|^{2}\right] < 1.
\end{equation*}
\end{definition}

Having defined these concepts of stability for stochastic differential systems, we now turn to the study of these concepts for the developed method.

\subsection{$A$-stability}
This type of stability by definition is strongly associated with the deterministic part of the equation. Let us see that the developed method has this stability property.

\begin{proposition}
The stochastic TR-BDF2 is $A$-stable.
\end{proposition}

\begin{proof}
The stochastic TR-BDF2 applied to the stochastic differential equation \eqref{test1} is
\begin{equation*}
\begin{array}{l}
\displaystyle Y_{n + 1} = \gamma_{3}Y_{n + \gamma} + \left(1 - \gamma_{3}\right)Y_{n} + Y_{n + 1}z\gamma_{2} + \Delta W_{t_{n}, t_{n + \gamma}}\sigma\left(1 - z\gamma\right)\left(1 - \gamma_{3}\right) + I_{\left(1, 0\right), t_{n}, t_{n + \gamma}}\sigma\lambda\left(1 - \gamma_{3}\right)\\
\displaystyle \mathrel{\phantom{Y_{n + 1} = }}+ \Delta W_{t_{n + \gamma}, t_{n + 1}}\sigma\left(1 - z\gamma_{2}\right) + I_{\left(1, 0\right), t_{n + \gamma}, t_{n + 1}}\sigma\lambda,
\end{array}
\end{equation*}
\noindent with $Y_{n + \gamma}$ solution of
\begin{equation*}
Y_{n + \gamma} = Y_{n} + \frac{1}{2}\left(Y_{n} + Y_{n + \gamma}\right)z\gamma + \Delta W_{t_{n}, t_{n + \gamma}}\sigma\left(1 - \frac{1}{2}z\gamma\right) + I_{\left(1, 0\right), t_{n}, t_{n + \gamma}}\sigma\lambda.
\end{equation*}

\noindent Combining the previous equations and solving $Y_{n + 1}$:
\begin{equation*}
\begin{array}{l}
\displaystyle Y_{n + 1} = \frac{1}{1 - z\gamma_{2}}\left(\frac{\gamma_{3}}{1 - \frac{1}{2}z\gamma}\left(Y_{n}\left(1 + \frac{1}{2}z\gamma\right) + \Delta W_{t_{n}, t_{n + \gamma}}\sigma\left(1 - \frac{1}{2}z\gamma\right) + I_{\left(1, 0\right), t_{n}, t_{n + \gamma}}\sigma\lambda\right)\vphantom{ + \left(1 - \gamma_{3}\right)Y_{n} + \Delta W_{t_{n}, t_{n + \gamma}}\sigma\left(1 - z\gamma\right)\left(1 - \gamma_{3}\right) + I_{\left(1, 0\right), t_{n}, t_{n + \gamma}}\sigma\lambda\left(1 - \gamma_{3}\right) + \Delta W_{t_{n + \gamma}, t_{n + 1}}\sigma\left(1 - z\gamma_{2}\right) + I_{\left(1, 0\right), t_{n + \gamma}, t_{n + 1}}\sigma\lambda}\right.\\
\displaystyle \mathrel{\phantom{Y_{n + 1} = }}+ \left(1 - \gamma_{3}\right)Y_{n} + \Delta W_{t_{n}, t_{n + \gamma}}\sigma\left(1 - z\gamma\right)\left(1 - \gamma_{3}\right) + I_{\left(1, 0\right), t_{n}, t_{n + \gamma}}\sigma\lambda\left(1 - \gamma_{3}\right)\\
\displaystyle \mathrel{\phantom{Y_{n + 1} = }}\left.\vphantom{\frac{\gamma_{3}}{1 - \frac{1}{2}z\gamma}\left(Y_{n}\left(1 + \frac{1}{2}z\gamma\right) + \Delta W_{t_{n}, t_{n + \gamma}}\sigma\left(1 - \frac{1}{2}z\gamma\right) + I_{\left(1, 0\right), t_{n}, t_{n + \gamma}}\sigma\lambda\right) + \left(1 - \gamma_{3}\right)Y_{n} + \Delta W_{t_{n}, t_{n + \gamma}}\sigma\left(1 - z\gamma\right)\left(1 - \gamma_{3}\right) + I_{\left(1, 0\right), t_{n}, t_{n + \gamma}}\sigma\lambda\left(1 - \gamma_{3}\right) }+ \Delta W_{t_{n + \gamma}, t_{n + 1}}\sigma\left(1 - z\gamma_{2}\right) + I_{\left(1, 0\right), t_{n + \gamma}, t_{n + 1}}\sigma\lambda\right).
\end{array}
\end{equation*}

\noindent Taking the values of $\gamma_{2}$ and $\gamma_{3}$, it can be written as:
\begin{equation*}
Y_{n + 1} = G\left(z\right)Y_{n} + V_{n} + V_{n + \gamma},
\end{equation*}
\noindent where
\begin{equation*}
G\left(z\right) = \frac{\left(1 + \left(1 - \gamma\right)^{2}\right)z + 2\left(2 - \gamma\right)}{\gamma\left(1 - \gamma\right)z^{2} + \left(\gamma^{2} - 2\right)z + 2\left(2 - \gamma\right)},
\end{equation*}
\begin{equation*}
V_{n} = \Delta W_{t_{n}, t_{n + \gamma}}\sigma\frac{\left(1 - \gamma\right)^{2}z + \left(2 - \gamma\right)}{\left(\gamma - 1\right)z + \left(2 - \gamma\right)} + I_{\left(1, 0\right), t_{n}, t_{n + \gamma}}\sigma\lambda\frac{\left(1 - \gamma\right)^{2}z + 2\left(2 - \gamma\right)}{\gamma\left(1 - \gamma\right)z^{2} + \left(\gamma^{2} - 2\right)z + 2\left(2 - \gamma\right)},
\end{equation*}
\begin{equation*}
V_{n + \gamma} = \Delta W_{t_{n + \gamma}, t_{n + 1}}\sigma + I_{\left(1, 0\right), t_{n + \gamma}, t_{n + 1}}\sigma\lambda\frac{\left(2 - \gamma\right)}{\left(\gamma - 1\right)z + \left(2 - \gamma\right)}.
\end{equation*}

\noindent Iterating, we obtain
\begin{equation*}
Y_{n + 1} = \left(G\left(z\right)\right)^{n + 1}Y_{0} + \sum_{k = 0}^{n}\left(G\left(z\right)\right)^{n - k}\left(V_{k} + V_{k + \gamma}\right).
\end{equation*}

\noindent Therefore, the fact that the extra step $\gamma$ has been added does not change the analysis to be performed; and it is equivalent if we had the scheme written as in \eqref{recur}.

\noindent Studying the region of absolute stability of the stochastic TR-BDF2 scheme \eqref{trbdf2} is the same as studying its counterpart of the deterministic TR-BDF2 scheme \eqref{trbdf2deter}, which is $A$-stable. Thus concluding the proof.
\end{proof}

\subsection{$MS$-stability}
As mentioned above, for this case we consider the test equation \eqref{test2}:
\begin{equation*}
{\rm d}X_{t} = \lambda X_{t}{\rm d}t + \sigma X_{t}{\rm d}W_{t}.
\end{equation*}

This equation has a known analytical solution and is given by
\begin{equation*}
X_{t} = e^{\left(\lambda - \frac{\sigma^{2}}{2}\right)t + \sigma W_{t}}X_{0},
\end{equation*}
\noindent and the second moment is
\begin{equation}
E\left[X_{t}^{2}\right] = e^{\left(2\lambda + \sigma^{2}\right)t}X_{0}^{2}.
\label{19}
\end{equation}

As in the first stability case presented, we are interested in studying a particular scenario; namely when
\begin{equation}
2\Re\left(\lambda\right) + \left|\sigma\right|^{2} < 0.
\label{condicion}
\end{equation}

When this inequality holds, the second moment \eqref{19} tends to zero as $t$ goes to infinity, and we want to study when the numerical scheme will behave similarly under that condition. On the other hand, note that it also generalizes the deterministic case, wherein the inequality \eqref{condicion} would be $\Re\left(\lambda\right) < 0$ if $\sigma = 0$.

\begin{proposition}\label{prop}
There exists $h^{\ast} > 0$ such that for all $h < h^{\ast}$, the stochastic TR-BDF2 is $MS$-stable.
\end{proposition}

\begin{proof}
When we calculate $E\left[\left|R\left(h, \lambda, \sigma, I_{\alpha}\right)\right|^{2}\right] - 1$ for the stochastic TR-BDF2 method, it becomes a polynomial, denoted by $P$, of degree seven in the variable $h$, so when $h$ tends to zero, $P$ converges to the independent term, which in this case is $2\Re\left(\lambda\right) + \left|\sigma\right|^{2}$.

\noindent Applying the definition of continuity, we obtain that, for all $\varepsilon > 0$, there exists $\delta \coloneqq h^{\ast} > 0$ such that, for all $h \in \left[0, h^{\ast}\right)$, the following inequality holds
\begin{equation*}
\left|P\left(h\right) - \left(2\Re\left(\lambda\right) + \left|\sigma\right|^{2}\right)\right| < \varepsilon,
\end{equation*}
\noindent or, equivalently,
\begin{equation*}
- \varepsilon + 2\Re\left(\lambda\right) + \left|\sigma\right|^{2} < P\left(h\right) < \varepsilon + 2\Re\left(\lambda\right) + \left|\sigma\right|^{2}.
\end{equation*}

\noindent Then, if we take $\varepsilon < - \left(2\Re\left(\lambda\right) + \left|\sigma\right|^{2}\right)$, $P\left(h\right) < 0$, $\forall h \in \left[0, h^{\ast}\right)$ and consequently the scheme is $MS$-stable.
\end{proof}

As mentioned in the introduction of this section, many schemes require a very small time step $h$ to achieve accurate performance. In our case, the above result states that the stochastic TR-BDF2 scheme performs well, in the $MS$-stability sense, for a certain $h < h^{\ast}$.

In Figure~\ref{contour}, we present the values of $h^{\ast}$ for the test equation \eqref{test2} with parameters $\Re\left(\lambda\right) \in \left[- 200, 0\right)$ and $\left|\sigma\right| \in \left[0, 20\right)$, calculated numerically using the stochastic TR-BDF2 scheme \eqref{trbdf2}.

\begin{figure}[ht]
\centering
\includegraphics[scale=0.41]{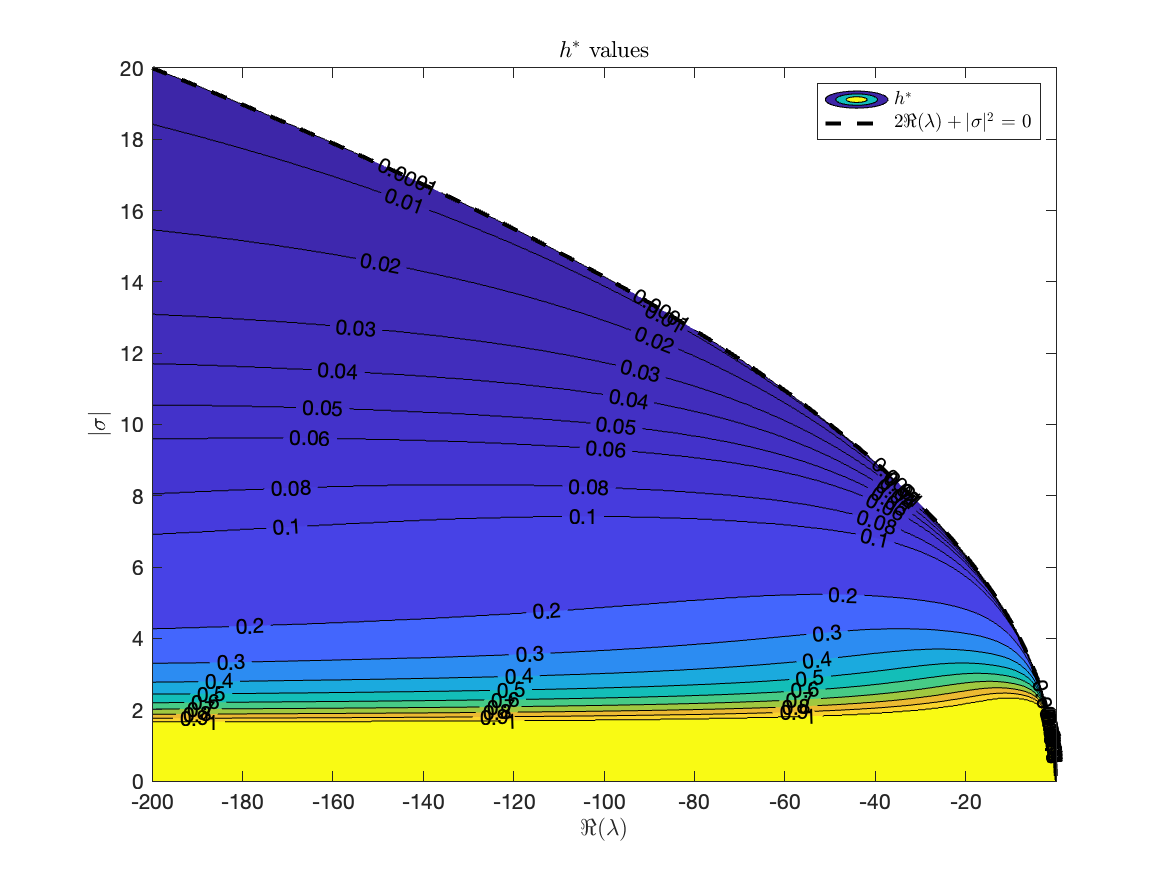}
\caption{$h^{\ast}$ values for the stochastic TR-BDF2 scheme.}\label{contour}
\end{figure}

It can be seen that the values of $h^{\ast}$ are acceptable in general: very large with small noise $\left(\left|\sigma\right| < 2\right)$ and decreasing as the parameters are closer to condition limit $2\Re\left(\lambda\right) + \left|\sigma\right|^{2} = 0$, would only be very small and therefore unfeasible when it is very close to the boundary condition.

We consider the case of the It{\^o}--Taylor method of order 2. In this situation the reasoning of Proposition~\ref{prop} can be exactly reproduced with the only difference that here the polynomial is of degree 3 and therefore the scheme is also $MS$-stable.

However, if we consider the second greatest term of the polynomials when $h$ tends to zero, the lineal one, for the It{\^o}--Taylor approximation of order 2 we have
\begin{equation*}
\frac{1}{2}\left(2\Re\left(\lambda\right) + \left|\sigma\right|^{2}\right)^{2},
\end{equation*}
\noindent while for the stochastic TR-BDF2
\begin{equation*}
\frac{1}{2 - \gamma}\left(2\Re\left(\lambda\right)\gamma\left(\gamma - 1\right) + \left|\sigma\right|^{2}\left(2 - \gamma\right)\right)\left(2\Re\left(\lambda\right) + \left|\sigma\right|^{2}\right).
\end{equation*}

The It{\^o}--Taylor one is positive, while the second is negative. This means that, there exists $h^{\ast\ast} > 0$ such that for all $h \in \left[h^{\ast\ast}, h^{\ast}\right)$, the stochastic TR-BDF2 is $MS$-stable, whilst the It{\^o}--Taylor approximation of order 2 is not.

In the following section, we will study a stiff example where this phenomenon occurs. Also, Figure~\ref{IT2} shows how, indeed, the values of $h^{\ast}$ are smaller across the entire graph, meaning a smaller time step would be needed to obtain good results.

\begin{figure}[ht]
\centering
\includegraphics[scale=0.41]{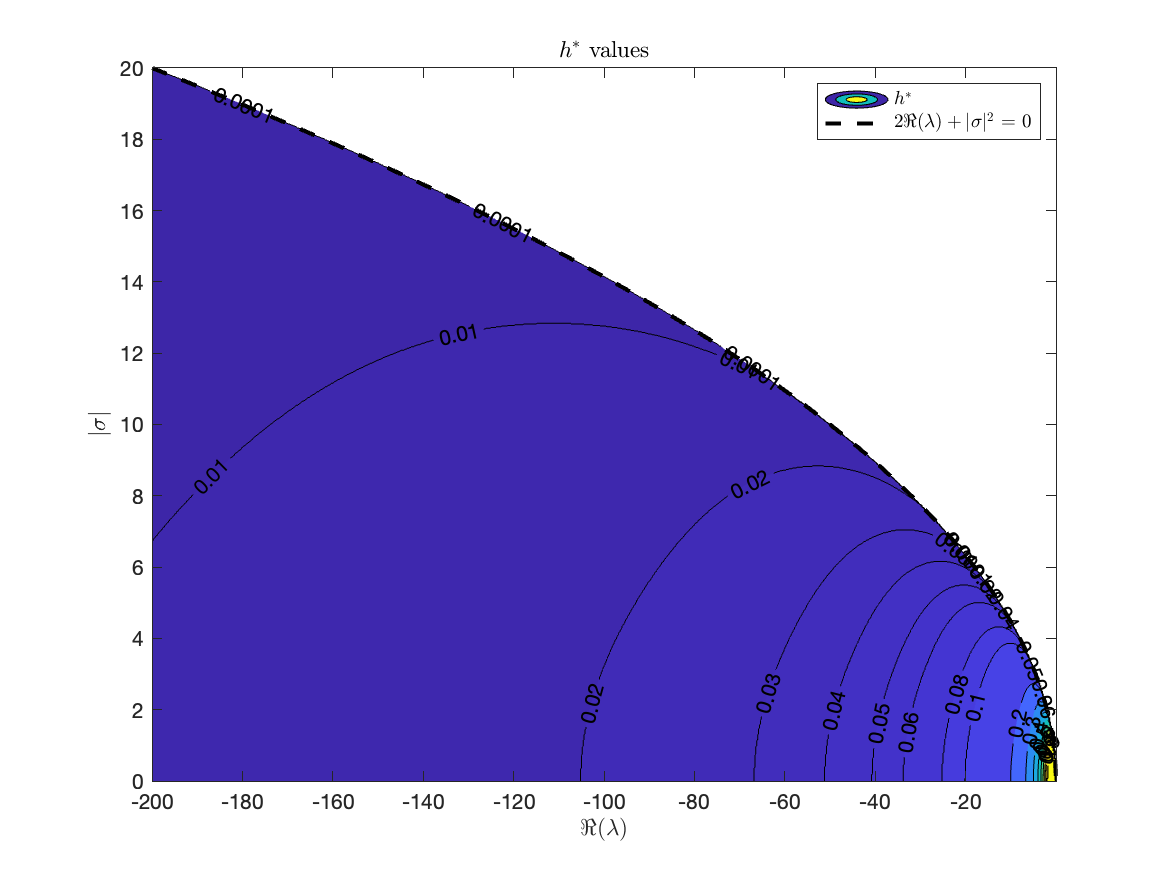}
\caption{$h^{\ast}$ values for the stochastic It{\^o}--Taylor of order 2 scheme.}\label{IT2}
\end{figure}

\section{Numerical validation}\label{numerical}
In this section we want to validate the theoretical results of the previous sections. On the one hand, we want to check the convergence orders of the stochastic TR-BDF2 and the different methods presented and denoted by TR-BDF2$^{r}$ (Subsection~\ref{sec41}) and, on the other hand, the behavior against another method of order 2 in the case of stiff problems. For this reason, we have selected three academic examples that are sufficiently illustrative to verify numerically the results that were established theoretically. For such examples, it is not necessary to use a method of order 2 with the stability properties of our schemes. However, our goal is to apply it in future work to problems where a more accurate approximation of the solution is needed, or where good stability properties are essential, such as in stiff problems.

To implement the method, it is necessary to compute all the stochastic integrals up to second order, which involves a laborious implementation. However, the same integrals appear in the It{\^o}--Taylor approximation of order 2, so we follow the techniques developed by Kloeden \cite{Kloeden1992}. For this reason, we plan to study how to optimise the computation of these stochastic integrals in future work, including, for example, the computation of L{\'e}vy areas, as done in \cite{Lord2014}.

Before starting the presentation of the results, we are going to detail some technical aspects that have been necessary for the implementation of the resolution algorithm:

\begin{itemize}
\item All tests were programmed using MATLAB$\_$R2022b software on a personal computer MacBook Air (M1, 8GB, 2020).

\item As this is an implicit method in each of its steps, it is necessary to use a method for solving non-linear equations. In this case, the Newton-Raphson method has been chosen because of its order 2 behavior in the error. If a lower order method is used, it may influence the order estimates of the total method.

\item In order to calculate the $L^{2}$ error numerically to check the algorithms used, we shall use the usual estimation presented in \cite{Lord2014}:
\begin{equation*}
\left\|X_{T} - Y_{N}\right\|_{L^{2}\left(\Omega, \mathbb{R}^{d}\right)} \simeq \left(\frac{1}{M}\sum_{j = 1}^{M}\left\|X_{T}^{j} - Y_{N}^{j}\right\|^{2}\right)^{\frac{1}{2}} \eqqcolon \varepsilon_{h},
\end{equation*}
\noindent where we approximate the error with the mean value error at the final time of $M$ simulations of sample paths of the Wiener process, fixed the time step $h$.

\item To estimate the convergence order, we compute the following expression for two values of $h$:
\begin{equation*}
p \simeq \frac{\log\left(\frac{\varepsilon_{h_{2}}}{\varepsilon_{h_{1}}}\right)}{\log\left(\frac{h_{2}}{h_{1}}\right)}.
\end{equation*}

In the particular case when $h_{1} = \frac{h}{2}$ and $h_{2} = h$ and the partition is uniform, we have
\begin{equation*}
p \simeq \log_{2}\left(\frac{\varepsilon_{h}}{\varepsilon_{\frac{h}{2}}}\right) \eqqcolon \rho_{h}.
\end{equation*}
\end{itemize}

\subsection{Test 1: Validation of convergence order}
To perform this validation we consider a regular problem where an analytical expression of the solution is known. Concretely, we consider the following SDE
\begin{equation*}
\left\{\begin{array}{ll}
\displaystyle {\rm d}X_{t} = - \beta^{2}X_{t}\left(1 - X_{t}^{2}\right){\rm d}t + \beta\left(1 - X_{t}^{2}\right){\rm d}W_{t}, & \displaystyle t \in \left[0, T\right],\\
\displaystyle X_{0} \text{ is given}, & \displaystyle
\end{array}\right.
\end{equation*}
\noindent where $\beta$ is a parameter and whose exact solution is given by
\begin{equation*}
X_{t} = \frac{\left(1 + X_{0}\right)e^{2\beta W_{t}} + X_{0} - 1}{\left(1 + X_{0}\right)e^{2\beta W_{t}} + 1 - X_{0}}.
\end{equation*}

We choose, the following values for the parameters: $h = 2^{- i}$, with $i = 3, \ldots, 10$, $\beta = \frac{1}{2}$, $X_{0} = \frac{1}{2}$, $T = 1$, $d = 1$, $m = 1$ and $M = 10000$.

\begin{figure}[ht]
\centering
\includegraphics[scale=0.41]{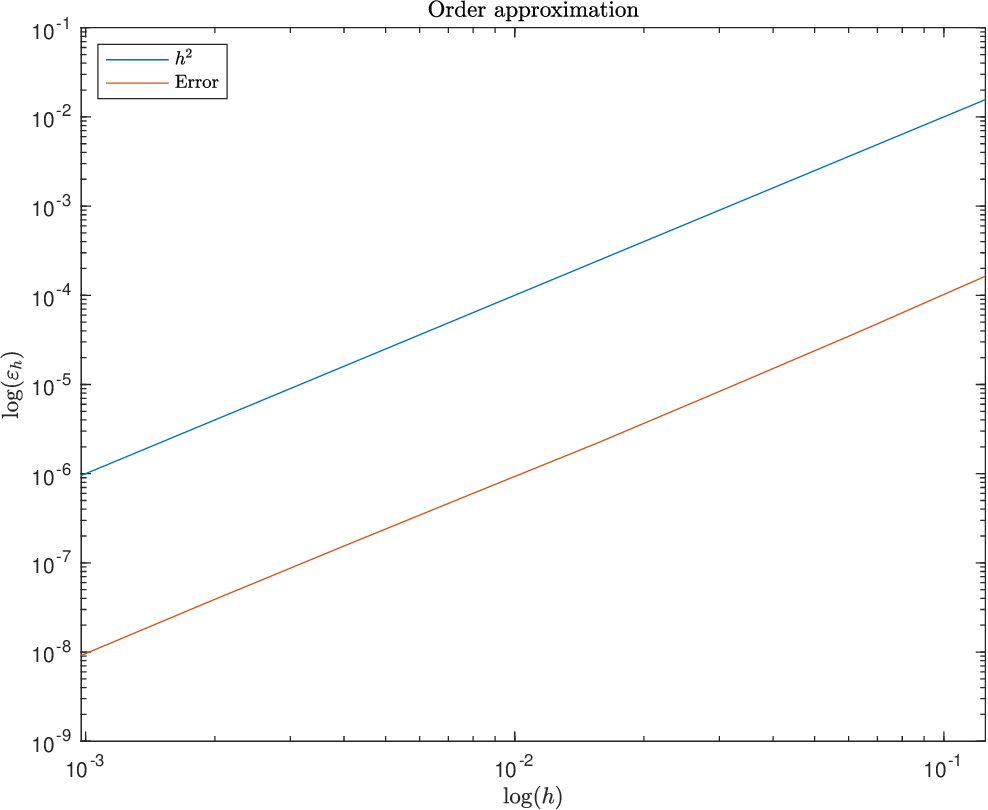}
\caption{Slope of the error of the stochastic TR-BDF2.}\label{fig1}
\end{figure}

Figure~\ref{fig1} shows the slope of the error curves for the method TR-BDF2. We obtain an excellent agreement with the theoretical predictions.

The following table shows the approximation of the convergence error for the different TR-BDF2$^{r}$ methods.

\begin{table}[ht]
\centering
\caption{Approximation of the order for the stochastic TR-BDF2 and the schemes TR-BDF2$^{r}$.}\label{w/o}
\begin{tabular}{@{}llllllll@{}}
\hline
Method & $\rho_{2^{- 3}}$ & $\rho_{2^{- 4}}$ & $\rho_{2^{- 5}}$ & $\rho_{2^{- 6}}$ & $\rho_{2^{- 7}}$ & $\rho_{2^{- 8}}$ & $\rho_{2^{- 9}}$\\ \hline
TR-BDF2$^{1}$ & 0.9992 & 1.0065 & 0.9924 & 1.0039 & 0.9836 & 0.9846 & 1.0160\\
TR-BDF2$^{2}$ & 1.6139 & 1.5487 & 1.5343 & 1.4640 & 1.5127 & 1.5082 & 1.5144\\
TR-BDF2$^{3}$ & 1.0117 & 1.0028 & 0.9914 & 1.0033 & 0.9839 & 0.9845 & 1.0158\\
TR-BDF2$^{4}$ & 1.6811 & 1.6105 & 1.5567 & 1.4803 & 1.5199 & 1.5116 & 1.5185\\
TR-BDF2$^{5}$ & 1.0295 & 1.0178 & 1.0121 & 0.9904 & 0.9849 & 1.0348 & 0.9653\\
TR-BDF2$^{6}$ & 1.5776 & 1.5465 & 1.5004 & 1.5135 & 1.5165 & 1.5119 & 1.4933\\
TR-BDF2 & 2.1159 & 2.0553 & 2.0343 & 1.9474 & 1.9636 & 1.9873 & 2.0202\\
\hline
\end{tabular}
\end{table}

As we can observe in Table~\ref{w/o} only the TR-BDF2 method developed with all its terms is the one that maintains the order 2 as in the deterministic one. In all the other cases we see that the order obtained is smaller.

\subsection{Test 2: Stiff problem}
In this test we will compare the developed method with the It{\^o}--Taylor method of order 2, for a stiff problem, in order to see the time step constraints to obtain an acceptable error accuracy.

Again, we consider a problem with a known solution but in this case with stiff behavior. We consider the following linear SDE
\begin{equation*}
\left\{\begin{array}{ll}
\displaystyle {\rm d}X_{t} = AX_{t}{\rm d}t + BX_{t}{\rm d}W_{t}, & \displaystyle t \in \left[0, T\right],\\
\displaystyle X_{0} \text{ is given}, & \displaystyle
\end{array}\right.
\end{equation*}
\noindent whose exact solution is given by
\begin{equation*}
X_{t} = e^{\left(A - \frac{1}{2}B^{2}\right)t + BW_{t}}X_{0}.
\end{equation*}

We take for this problem
\begin{equation*}
A = \begin{pmatrix}
\displaystyle - a & \displaystyle a\\
\displaystyle a & \displaystyle - a
\end{pmatrix} \text{ and } B = \begin{pmatrix}
\displaystyle b & \displaystyle 0\\
\displaystyle 0 & \displaystyle b
\end{pmatrix}.
\end{equation*}

The Lyapunov exponents of the stochastic differential equation of this example are $\lambda_{1} = - \frac{1}{2}b^{2}$ and $\lambda_{2} = - \frac{1}{2}b^{2} - 2a$. The stiffness in the stochastic sense for this particular case will occur if $\lambda_{1} \gg \lambda_{2}$, for $a > 0$.

For this, we choose $a = 15$, $b = \frac{1}{2}$, $X_{0} = \begin{pmatrix}
\displaystyle 1, & \displaystyle 0
\end{pmatrix}^{T}$, $T = 1$ and $M = 10000$. Table~\ref{stiff} shows the errors calculated using both It{\^o}--Taylor approximation of order 2 and the stochastic TR-BDF2. It can be seen how, in this stiff example, a small time step is required for the first one to obtain good results, which is not true for the other one; in the case of TR-BDF2 for all values a small error is obtained which decreases as a function of step $h$ as expected.

\begin{table}[ht]
\centering
\caption{Errors for the stochastic TR-BDF2 and the It{\^o}--Taylor approximation of order 2.}\label{stiff}
\begin{tabular}{@{}llllllll@{}}
\hline
Method & $\varepsilon_{2^{- 1}}$ & $\varepsilon_{2^{- 2}}$ & $\varepsilon_{2^{- 3}}$ & $\varepsilon_{2^{- 4}}$ & $\varepsilon_{2^{- 5}}$ & $\varepsilon_{2^{- 6}}$ & $\varepsilon_{2^{- 7}}$\\ \hline
It{\^o}--Taylor & 6876.8 & 1.567 e$^{5}$ & 84135 & 0.109 & 2.198 e$^{- 6}$ & 5.407 e$^{- 7}$ & 1.421 e$^{- 7}$\\
TR-BDF2 & 0.025 & 0.001 & 1.209 e$^{- 5}$ & 3.423 e$^{- 6}$ & 1.079 e$^{- 6}$ & 3.5 e$^{- 7}$ & 1.174 e$^{- 7}$\\
\hline
\end{tabular}
\end{table}

Taking the same values with $d = 5$ and matrices $A$ and $B$ as
\begin{equation*}
A_{i, j} = \left(- 1\right)^{i + j - 1}a \text{ and } B = bI_{5},
\end{equation*}
\noindent where $I_{d}$ is the identity matrix of order $d$, we have that $\lambda_{1} = - \frac{1}{2}b^{2}$ and $\lambda_{5} = - \frac{1}{2}b^{2} - 5a$ as before. We show the results in Table~\ref{stiff2}.

\begin{table}[ht]
\centering
\caption{Errors for the stochastic TR-BDF2 and the It{\^o}--Taylor approximation of order 2.}\label{stiff2}
\begin{tabular}{@{}llllllll@{}}
\hline
Method & $\varepsilon_{2^{- 1}}$ & $\varepsilon_{2^{- 2}}$ & $\varepsilon_{2^{- 3}}$ & $\varepsilon_{2^{- 4}}$ & $\varepsilon_{2^{- 5}}$ & $\varepsilon_{2^{- 6}}$ & $\varepsilon_{2^{- 7}}$\\ \hline
It{\^o}--Taylor & 1.988 e$^{5}$ & 2.795 e$^{8}$ & 1.154 e$^{12}$ & 2.991 e$^{13}$ & 25345 & 6.84 e$^{- 7}$ & 1.798 e$^{- 7}$\\
TR-BDF2 & 0.006 & 0.0003 & 1.537 e$^{- 5}$ & 4.33 e$^{- 6}$ & 1.365 e$^{- 6}$ & 4.427 e$^{- 7}$ & 1.485 e$^{- 7}$\\
\hline
\end{tabular}
\end{table}

We can conclude that the developed method is quite suitable for solving stiff problems without having to take an extremely small step $h$.

\subsection{Test 3: General problem}
In this test we will compare the developed method with the It{\^o}--Taylor method of order 2, for a general stochastic differential equation. We consider the $d \times d$ fundamental matrix $\Phi_{t}$ that satisfies $\Phi_{0} = I_{d}$ and the homogeneus matrix stochastic differential equation

\begin{equation*}
\left\{\begin{array}{ll}
\displaystyle {\rm d}\Phi_{t} = A\Phi_{t}{\rm d}t + \sum_{j = 1}^{m}B^{j}\Phi_{t}{\rm d}W_{t}^{j}, & \displaystyle t \in \left[0, T\right],\\
\displaystyle \Phi_{0} = I_{d}, & \displaystyle
\end{array}\right.
\end{equation*}
\noindent which we interpret column vector by column vector as vector stochastic differential equations and whose exact solution is given by
\begin{equation*}
\Phi_{t} = e^{\left(A - \frac{1}{2}\sum_{j = 1}^{m}\left(B^{j}\right)^{2}\right)t + \sum_{j = 1}^{m}B^{j}W_{t}^{j}},
\end{equation*}
\noindent if the matrices $A, B^{1}, B^{2}, \ldots, B^{m}$ are constants and commute, that is if $AB^{j} = B^{j}A$ and $B^{j_{1}}B^{j_{2}} = B^{j_{2}}B^{j_{1}}$ for all $j, j_{1}, j_{2} = 1, 2, \ldots, m$.

We choose: $h = 2^{- i}$, with $i = 1, \ldots, 7$, $T = 1$, $d = 2$, $m = 2$ and $M = 10000$,
\begin{equation*}
A = - 15\begin{pmatrix}
\displaystyle 1 & \displaystyle 0\\
\displaystyle 0 & \displaystyle 1
\end{pmatrix}, B^{1} = \frac{1}{4}\begin{pmatrix}
\displaystyle 2 & \displaystyle 0\\
\displaystyle 0 & \displaystyle 0
\end{pmatrix} \text{ and } B^{2} = \frac{1}{4}\begin{pmatrix}
\displaystyle 0 & \displaystyle 0\\
\displaystyle 0 & \displaystyle 1
\end{pmatrix}.
\end{equation*}

We rewrite the problem as a $4$-dimensional SDE with a $2$-dimensional noise as follows:
\begin{equation*}
\left\{\begin{array}{ll}
\displaystyle {\rm d}X_{t} = - 15X_{t}{\rm d}t + \frac{1}{2}\begin{pmatrix}
\displaystyle I_{2} & \displaystyle 0_{2 \times 2}\\
\displaystyle 0_{2 \times 2} & \displaystyle 0_{2 \times 2}
\end{pmatrix}X_{t}{\rm d}W_{t}^{1} + \frac{1}{4}\begin{pmatrix}
\displaystyle 0_{2 \times 2} & \displaystyle 0_{2 \times 2}\\
\displaystyle 0_{2 \times 2} & \displaystyle I_{2}
\end{pmatrix}X_{t}{\rm d}W_{t}^{2}, & \displaystyle t \in \left[0, T\right],\\
\displaystyle X_{0} = \begin{pmatrix}
\displaystyle 1, & \displaystyle 0, & \displaystyle 0, & \displaystyle 1
\end{pmatrix}^{T}, & \displaystyle
\end{array}\right.
\end{equation*}
\noindent where the matrix of the diffusion is given by blocks and $0_{i \times j}$ the null matrix with $i$ rows and $j$ columns.

In Table~\ref{stiff3} we show the phenomena described in the above Subsection and we calculate an approximation of the order as in Table~\ref{w/o} for the stochastic TR-BDF2 in Table~\ref{w/o3}.

\begin{table}[ht]
\centering
\caption{Errors for the stochastic TR-BDF2 and the It{\^o}--Taylor approximation of order 2.}\label{stiff3}
\begin{tabular}{@{}llllllll@{}}
\hline
Method & $\varepsilon_{2^{- 1}}$ & $\varepsilon_{2^{- 2}}$ & $\varepsilon_{2^{- 3}}$ & $\varepsilon_{2^{- 4}}$ & $\varepsilon_{2^{- 5}}$ & $\varepsilon_{2^{- 6}}$ & $\varepsilon_{2^{- 7}}$\\ \hline
It{\^o}--Taylor & 666.45 & 490.92 & 0.561 & 2.256 e$^{- 5}$ & 5.257 e$^{- 7}$ & 8.05 e$^{- 8}$ & 1.732 e$^{- 8}$\\
TR-BDF2 & 0.063 & 0.0004 & 4.601 e$^{- 7}$ & 2.185 e$^{- 7}$ & 6.156 e$^{- 8}$ & 1.565 e$^{- 8}$ & 3.906 e$^{- 9}$\\
\hline
\end{tabular}
\end{table}

\begin{table}[ht]
\centering
\caption{Approximation of the order for the stochastic TR-BDF2.}\label{w/o3}
\begin{tabular}{@{}llll@{}}
\hline
Method & $\rho_{2^{- 4}}$ & $\rho_{2^{- 5}}$ & $\rho_{2^{- 6}}$\\ \hline
TR-BDF2 & 1.8279 & 1.9755 & 2.0025\\
\hline
\end{tabular}
\end{table}

\paragraph*{Acknowledgement}
The authors would like to thank Prof. Luca Bonaventura, Politecnico di Milano, for his suggestions, comments and discussions that have allowed the improvement of this work. This document is the result of the research project funded by the Andalusian Government under grant PREDOC-PAID2020 and the Spanish Ministerio de Ciencia e Innovaci{\'o}n, Agencia Estatal de Investigaci{\'o}n (AEI) and FEDER under projects PID2024-156228NB-I00 and PID2021-123153OB-C21.

\bibliographystyle{plain}
\bibliography{main}
\end{document}